\PassOptionsToPackage{dvipsnames}{xcolor}
\documentclass[journal,twoside,web]{ieeecolor}
\newif\ifBio
\Biofalse
\newif\ifReview
\Reviewfalse
\input{Sources/header.tex}
\newcommand{\Rpp}{\R_{>0}}
\newcommand{\Np}{\N_{>0}}
\newcommand{\R}{\mathbb{R}}
\newcommand{\Rx}{\mathbb{R}^{n_x}}
\newcommand{\Rp}{\mathbb{R}^{n_p}}

\newcommand{\Ru}{\mathbb{R}^{n_u}}
\newcommand{\Z}{\mathbb{Z}}
\newcommand{\N}{\mathbb{N}}
\newcommand{\PPP}{\mathscr{Y}}
\newcommand{\F}{\mathcal{F}}
\newcommand{\J}{\mathcal{J}}
\renewcommand{\L}{\mathscr{L}}
\newcommand{\p}{\bar{p}}
\newcommand{\PP}{\mathscr{P}}

\ifReview
    
    \newcommand{\rz}[1]{\textcolor{NavyBlue}{#1}}
\else
    
    \newcommand{\rz}[1]{#1}
\fi

\newcommand{\QQ}{Q}
\newcommand{\qq}{q}
\newcommand{\FF}{F}
\newcommand{\f}{\varphi}
\newcommand{\g}{\psi}
\newcommand{\ff}{\phi}
\newcommand{\GG}{G}
\renewcommand{\gg}{g}
\newcommand{\HH}{H}
\newcommand{\hh}{h}
\begin{document}
%
%
\title{BP-MPC: Optimizing the Closed-Loop Performance of MPC using BackPropagation}
\author{Riccardo Zuliani, Efe C. Balta, and John Lygeros
\thanks{Research supported by the Swiss National Science Foundation under NCCR Automation (grant agreement 51NF40\_180545). R. Zuliani and J. Lygeros are with the Automatic Control Laboratory (IfA), ETH Z\"urich, 8092 Z\"urich, Switzerland \texttt{\small$\{$rzuliani,lygeros$\}$@ethz.ch}. E. C. Balta is with Inspire AG, 8005 Z\"urich, Switzerland \& with IfA \texttt{\small efe.balta@inspire.ch}.}}
\maketitle
\begin{abstract}
Model predictive control (MPC) is pervasive in research and industry. However, designing the cost function and the constraints of the MPC to maximize closed-loop performance remains an open problem. To achieve optimal tuning, we propose a backpropagation scheme that solves a policy optimization problem with nonlinear system dynamics and MPC policies. We enforce the system dynamics using linearization and allow the MPC problem to contain elements that depend on the current system state and on past MPC solutions. Moreover, we propose a simple extension that can deal with losses of feasibility. Our approach, unlike other methods in the literature, enjoys convergence guarantees.
\end{abstract}
\section{Introduction}
%
%
Among optimization-based control schemes, \textit{model predictive control} (MPC) has recently attracted increasing attention in both industry and academia. This technique enables feedback by repeatedly solving a numerical optimization problem at every time-step, each time taking into account the current (measured or estimated) state of the system as well as process and input constraints. Because of its effectiveness in practical applications, researchers have dedicated significant effort to the task of designing MPC controllers. For example, \cite{chen1998quasi} showed that the introduction of an appropriately selected terminal cost can ensure stability and feasibility of the closed-loop. More recently, \cite{di2009model} proposed a design to ensure that the MPC behaves like a linear controller around a specified operating point, with the goal of inheriting the well-known stability and robustness properties of linear controllers. The objective function of an MPC can also be chosen to incentivize learning of an unknown model, as proposed in \cite{soloperto2020augmenting}.

MPC design can be viewed as a policy optimization problem. \textit{Policy optimization} is a well-known problem in reinforcement learning, where the goal is to obtain a control policy that minimizes some performance objective \cite{silver2014deterministic}. In common applications, the policy is parameterized by problem parameters, states, or inputs, and gradient-based techniques are used to learn the optimal parameters. In the context of MPC, the design parameters are generally the cost and the constraints of the problem. The challenge when considering model predictive control policies is that the MPC policy and resulting closed-loop performance are generally not differentiable with respect to the parameters.

Recently, differentiable optimization provided a principled way to overcome the nondifferentiability issue. Specifically, \cite{amos2017optnet} proved that, under certain conditions, the optimizer of a quadratic program (QP) is indeed continuously differentiable with respect to design parameters appearing in the cost and the constraints, and that the gradient can be retrieved by applying the implicit function theorem \cite{dontchev2009implicit} to the KKT conditions of the QP. Since MPC problems are often formulated as QPs, this approach effectively allows for the differentiation of MPC policies. This discovery led to a plethora of applications of differentiable optimization in the realm of model predictive control. For example, \cite{amos2018differentiable} considers the problem of imitation learning, where the tuning parameters are the cost and the model of the linear dynamics of an MPC problem. The idea of utilizing the KKT conditions to obtain derivatives of an optimization problem does not stop with quadratic programs. \cite{oshin2023differentiable} uses the same technique to compute gradients of a nonlinear optimal control problem, and uses this information to conduct online design of a robust model predictive controller. The goal in this case is to match the performance of a nominal controller. Similarly, \cite{cortez2023robust} introduces a predictive safety filter to ensure the safety of the closed-loop operation.

\rz{In a similar fashion, over the course of several papers \cite{gros2019data,zanon2020safe,zanon2020reinforcement}, Gros and Zanon used policy gradient methods to optimize the performance of nonlinear economic MPC. The main idea behind these works is to utilize a nonlinear MPC as a function approximator that can encode both the value and the action-value functions of a given problem. Their algorithm can produce MPC schemes that are stabilizing by construction \cite{gros2019data} and safe against additive disturbances (only in the case of affine systems) \cite{zanon2020safe}. In particular, \cite{zanon2020reinforcement} is the work that we believe is closest to ours, as it uses a linearization-based procedure to avoid solving a nonconvex problem online. However, the authors focus on infinite horizon problems and do not provide a detailed treatment of the convergence of the optimization algorithm. Additionally, the parameter update is performed online (i.e. as the controller is deployed on the system).}

These sophisticated differentiable optimization-based methods rely on the assumption that the optimizer of the MPC problem is continuously differentiable, since the gradient of the optimizer is obtained using the implicit function theorem. It is well known, however, that this may not be the case, and that the optimizer may not be everywhere differentiable even for simple projection problems \cite{bednarczuk2021lipschitz}. The continuous differentiability assumption can be relaxed thanks to the recently developed concept of conservative Jacobians \cite{bolte2021conservative}. \textit{Conservative Jacobians} are set-valued operators that extend the concept of gradients to almost-everywhere differentiable functions. \rz{Similarly to other generalized Jacobians, they obey the chain rule of differentiation and can be used to create first-order optimization schemes with convergence guarantees \cite{davis2020stochastic}. However, unlike e.g. Clarke Jacobians, conservative Jacobians satisfy a nonsmooth implicit function theorem, which is essential in our setting to obtain the sensitivity of the solution maps of the MPC problems \cite{bolte2021nonsmooth}.}

\rz{In this paper, we consider the problem of optimizing the closed-loop trajectory directly by backpropagation. Specifically, we compute the conservative Jacobian of the entire closed-loop trajectory with respect to variations of the design parameters by applying the chain rule to the conservative Jacobians of each MPC problem. We then apply a gradient-based scheme to update the value of the parameter and obtain improved closed-loop performance. This is fundamentally different than optimizing a single MPC step as it accounts for the effect of the receding horizon, where past decisions influece future ones.}

The idea of using backpropagation to improve closed-loop performance of MPC first appeared in \cite{agrawal2020learning} and \cite{agrawal2019differentiable}. These studies focused on linear dynamics without state constraints and did not provide formal convergence guarantees. Our work makes the following contributions.
\begin{enumerate}
    \item We utilize the backpropagation paradigm to solve a nonconvex closed-loop policy optimization problem where the policy is a parameterized MPC. The MPC utilizes a linearized version of the system dynamics to retain convexity.
    \item \rz{We provide conditions under which the closed-loop optimization problem is well posed by extending \cite{amos2017optnet} to the nonsmooth regime, and propose a gradient-based method with convergence guarantees (to a critical point).}
    \item We allow the MPC to have cost and constraints that depend on the current state of the system and on the solution of the MPC problem in the previous time-steps, allowing, for example, the application of the real-time iteration \cite{diehl2005real}.
    \item We propose a simple extension to deal with cases where the MPC scheme loses feasibility and provide conditions under which the closed-loop is guaranteed to converge to a safe operation.
\end{enumerate}

To compute the conservative Jacobian of each optimization problem, we adapt and extend the techniques described in \cite{bolte2022differentiating} to a control theoretic context. Additionally, we derive problem-specific sufficient conditions under which the nonsmooth implicit function theorem in \cite{bolte2022differentiating} can be applied. We finally showcase our findings through simulation on a nonlinear problem.

\rz{Our algorithm can be applied under the assumption that the initial condition of the system is known (this is the case e.g. for iterative control tasks). The work presented in this paper has been recently extended in \cite{zuliani2024closed} to uncertain systems subject to additive noise and with uncertain initial conditions.}

The remainder of this paper is structured as follows. \cref{section:PF} describes the system dynamics, the control policy, and the policy optimization problem. \cref{section:CJ} presents a short recap of conservative Jacobians, their main calculus rules, and a way to minimize such functions with a first-order scheme. \cref{section:DIFF} demonstrates how the conservative Jacobian of an MPC problem can be computed. \cref{section:CLopt} showcases our main algorithmic contribution by describing the backpropagation scheme and the main optimization algorithm. In \cref{section:EXT} we provide some useful extensions to our scheme, such as nonlinear dynamics and recovery from infeasibility. In \cref{section:SIM} we showcase our methods in simulation.

\subsubsection*{Notation}

We use $\N$, $\Z$, $\R$ to denote the set of natural, integer, and real numbers, respectively.
$\Z_{[a,b]}$ is the set of integers $z$ with $a \leq z \leq b$, for some $a \leq b$.
If $C \subset \R^n$ is a convex set, we denote with $P_C$ the orthogonal projector to the set.
Given a matrix $A\in\R^{n \times m}$, we use $\mathbf{r}(A,j)$ to denote the $j$-th row of $A$ (with $j\in\Z_{[1,n]}$). 
We use $A \succ 0 $ ($A \succeq 0$) to indicate that the symmetric matrix $A$ is positive definite (positive semi-definite). We use $A \sim B$ to indicate that $A$ is a function of $B$. $\|\cdot\|$ denotes the $2$-norm, \rz{and $\langle a,b \rangle=a^\top b$ is the Euclidean inner product}.
\section{Problem formulation}\label{section:PF}
\subsection{System dynamics and constraints}\label{subsection:PF_sys}
We consider a nonlinear time-invariant system where the state dynamics are given for each time-step $t\in\N$ by
\begin{align}
\label{eq:PF_system}
\bar{x}_{t+1}= f(\bar{x}_t,\bar{u}_t),
\end{align}
with $f$ locally Lipschitz and $\bar{x}_0\in\Rx$ known. We assume that $(\bar{x}_t,\bar{u}_t)=0$ is an equilibrium for \cref{eq:PF_system}. The \textit{state} $\bar{x}_t\in\Rx$ and the \textit{input} $\bar{u}_t\in\Ru$ are subject to polytopic constraints
\begin{align}\label{eq:PF_constraints}
H_x \bar{x}_t \leq h_x, ~~
H_u \bar{u}_t \leq h_u.
\end{align}
The control input $\bar{u}_t$ is determined, at each time-step, by a parameterized control policy $\pi:\R^{n_x} \times \R^{n_p}\to \R^{n_u}$
\begin{align}
\label{eq:PF_policy}
\bar{u}_t = \pi(\bar{x}_t,p).
\end{align}
The \textit{parameter vector} $p\in\R^{n_p}$ parameterizes the control policy $\pi$ at any state $\bar{x}_t$ \rz{(we refer the reader to \cref{subsection:PF_mpc} for a concrete example of $p$ in the context of MPC)}. We require $p$ to satisfy the constraint $p\in \PP$, for some polytopic set $\PP$. Below, we restrict attention to MPC control policies.

The goal of this paper is to minimize an objective function involving $p$ and the closed-loop state and input trajectory $ (\bar{x},\bar{u}):=(\bar{x}_0,\dots,\bar{x}_{T+1},\bar{u}_0,\dots,\bar{u}_{T})$ for some finite time interval $T\in\N_{>0}$, under the constraints in \cref{eq:PF_constraints}.
\begin{equation}\label{eq:PF_problem}
\begin{split}\begin{aligned}
\operatorname*{minimize}_{\bar{x},\bar{u},\p\in \PP} & \quad \mathcal{C}(\bar{x},\bar{u},p) \\
\mathrm{subject~to} & \quad \bar{x}_{t+1}= f(\bar{x}_t,\bar{u}_t), ~ \bar{x}_0 \text{ given},\\
& \quad \bar{u}_t = \pi(\bar{x}_t,p),\\
& \quad H_x \bar{x}_t\leq h_x, ~ H_u \bar{u}_t\leq h_u,~t \in \Z_{[0,T]},
\end{aligned}
\end{split}
\end{equation}
where $\mathcal{C}:\R^{(T+2)n_x} \times \R^{(T+1)n_u} \times \R^{n_p} \to \R_{\geq 0}$ specifies the performance objective. In \cref{eq:PF_problem}, $T$ should be chosen large enough to reach the desired equilibrium condition. Note that problem \cref{eq:PF_problem} may be non-convex. In the following, for simplicity, we consider the case
\begin{align}\label{eq:PF_norm_cost_cl}
\mathcal{C}(\bar{x},\bar{u},p) = \mathcal{C}(\bar{x}) = \sum_{t=0}^{T} \|\bar{x}_t\|_{Q_x}^2
\end{align}
for some $Q_x\in\R^{n_x \times n_x}$ with $Q_x \succ 0$. Our method can easily be extended to more general cost functions as described in \cref{subsection:EXT_nonconvex_cost}.

\subsection{Model predictive control}\label{subsection:PF_mpc}
In this paper, we restrict attention to MPC policies, where the control input is chosen as the solution of an optimal control problem.  Specifically, after measuring the current state $\bar{x}_t$, we use the knowledge we possess about the system \cref{eq:PF_system} to optimize the future prediction of the state-input trajectories of the system. The predicted trajectories are denoted by $x_t:=(x_{0|t},\dots,x_{N|t})\in\R^{(N+1)n_x}$ and $u_t:=(u_{0|t},\dots,u_{N-1|t})\in\R^{Nn_u}$, where $N\in\Np$, with $N\ll T$, is the \textit{prediction horizon} of the MPC. The initial state is chosen to be equal to the true state of the system, $x_{0|t}=\bar{x}_t$ and, to ensure convexity, we approximate the state dynamics as
\begin{align*}
x_{k+1|t} = A_t x_{k|t} + B_t u_{k|t} + c_t,
\end{align*}
where $A_t$, $B_t$, and $c_t$ are known at runtime and should be chosen to accurately approximate the real dynamics \cref{eq:PF_system} in the vicinity of $\bar{x}_t$. We use $S_t:=(A_t,B_t,c_t)$ to compactly represent the approximate dynamics at time $t$.

Each predicted state and input must satisfy the constraints \cref{eq:PF_constraints}. In addition, we generally impose different constraints on the predicted terminal state $x_{N|t}$
\begin{align}
\label{eq:PF_MPC_terminal}
H_{x,N} x_{N|t} \leq h_N.
\end{align}

The objective function in the MPC is an approximation of the objective in \cref{eq:PF_problem}, given by
\begin{align*}
\|x_{N|t}\|_P^2 + \sum_{k=0}^{N-1} \|x_{k|t}\|_{Q_x}^2 + \|u_{k|t}\|_{R_u}^2,
\end{align*}
where we added a terminal penalty $\|x_{N|t}\|_P^2$ and a penalty on the input, with $P,R_u \succ 0$, to ensure that the problem is strongly convex. The MPC problem that is solved online at each time-step is therefore given as follows.
\begin{align}
\begin{split}\hypertarget{PF_mpc}{}
\operatorname*{min.}_{x_t,u_t} & \quad \|x_{N|t}\|_P^2 + \sum_{k=0}^{N-1} \|x_{k|t}\|_{Q_x}^2 + \|u_{k|t}\|_{R_u}^2\\
\mathrm{s.t.} & \quad x_{k+1|t}=A_tx_{k|t}+B_tu_{k|t}+c_t, ~ x_{0|t} = \bar{x}_t,\\
& \quad H_x x_{k|t} \leq h_x,~ H_u u_{k|t} \leq h_u, ~ k\in\Z_{[0,N-1]},\\
& \quad H_{x,N} x_{N|t} \leq h_{x,N},
\end{split}\label{eq:PF_mpc}
\end{align}
At each time-step, after measuring $\bar{x}_t$ and obtaining $S_t$, we solve \cref{eq:PF_mpc} and choose $\pi(\bar{x}_t,p)=u_{0|t}$, where $u_{0|t}$ is the first entry of the input trajectory. We use $\mathrm{MPC}(\bar{x}_t,S_t,p)$ to denote the function that maps a parameter $p$, a nominal system $S_t$, and an initial condition $\bar{x}_t$ to a control input $\bar{u}_t=u_{0|t}$, so that
\begin{align}\label{eq:PF_policy_mpc}
\pi(\bar{x}_t,p)=\mathrm{MPC}(\bar{x}_t,S_t,p).
\end{align}
Here we treat the terminal cost and the input cost as tunable parameters, by letting $p:=(P,R_u)$. However, with the same framework, one can also choose $p$ as any other element appearing in the cost or in the constraints of \cref{eq:PF_mpc}.

\subsection{A projected gradient-based framework}\label{subsection:PF_gd_framework}
For the time being, we assume that $S_t\equiv S$ and drop it from the notation; we deal with the more complex case where $S_t$ is determined online in \cref{subsection:PF_choosing_st}. Combining problem \cref{eq:PF_problem} with the cost function \cref{eq:PF_norm_cost_cl} and the controller \cref{eq:PF_policy_mpc}, leads to the closed-loop control problem
\begin{align}\label{eq:PF_problem_2}
\begin{split}
\operatorname*{minimize}_{\bar{x},p\in\PP} & \quad \sum_{t=0}^{T} \|\bar{x}_t\|_{Q_x}^2 \\
\mathrm{subject~to} & \quad \bar{x}_{t+1} = f(\bar{x}_t,\mpce),~\bar{x}_0 \text{ given},\\
& \quad H_x \bar{x}_t \leq h_x , ~ t\in\Z_{[0,T]},
\end{split}
\end{align}
Note that we can remove the input constraints in \cref{eq:PF_constraints}, as they are automatically satisfied if the inputs are obtained from the MPC \cref{eq:PF_mpc}. As shown in \cref{appendix:appA}, \cref{eq:PF_problem_2} can be compactly rewritten as follows.
\begin{align}
\begin{split}\label{eq:PF_closed_loop_compact_problem}
\operatorname*{minimize}_{p\in\PP} & \quad \mathcal{C}(\bar{x}(p))\\ 
\mathrm{subject~to}& \quad \mathcal{H}(\bar{x}(p)) \leq 0,
\end{split}
\end{align}
where $\bar{x}(p)$ is the closed-loop state trajectory generated by the dynamics \cref{eq:PF_system} under controller \cref{eq:PF_policy_mpc} for a given value of $p$. In the following section, we derive an efficient procedure to obtain generalized gradients of the function $\mathcal{C}$ with respect to $p$.
\section{Conservative Jacobians}\label{section:CJ}
In the upcoming sections, we repeatedly deal with the problem of minimizing a nonsmooth, nonconvex function. These problems admit a simple solution strategy based on a descent algorithm. However, because of the nonsmoothness, we cannot always guarantee the existence of a gradient. Luckily, we can still devise descent algorithms if the function is almost everywhere differentiable thanks to the concept of \textit{conservative Jacobian}. This section describes how conservative Jacobians generalize the notion of gradient to functions that are almost everywhere differentiable.

A \textit{path} is an absolutely continuous function $x:[0,1]\to\R^n$ which admits a derivative $\dot{x}$ for almost every $t\in[0,1]$, and for which $x(t)-x(0)$ is the Lebesgue integral of $\dot{x}$ between $0$ and $t$ for all $t\in[0,1]$.

\begin{definition}[\hspace{1sp}{\cite[Section 2]{bolte2021nonsmooth}}]
A locally Lipschitz function $\f:\R^n\to\R^m$ admits $\mathcal{J}_{\f}:\R^n\rightrightarrows\R^{m \times n}$ as a \textit{conservative Jacobian}, if $\mathcal{J}_{\f}$ is nonempty-valued, outer semicontinuous, locally bounded, and for all paths $x:[0,1]\to\R^n$ and almost all $t\in[0,1]$
\begin{align}\label{eq:CJ_chain_rule}
\odv{}{t}\f(x(t)) = \langle v,\dot{x}(t) \rangle ,~ \forall v\in \mathcal{J}_{\f}(x(t)).
\end{align}
A locally Lipschitz function $\f$ that admits a conservative Jacobian $\mathcal{J}_{\f}$ is called \textit{path-differentiable}.
\end{definition}
\rz{If $\f:\R^n \times \R^p \to \R^n$ is a function of two arguments $p$ and $x$, we define $\mathcal{J}_{\f,x}(\tilde{p},\tilde{x})=\{ V:[U~V]\in\J_{\f}(\tilde{p},\tilde{x}) \}$ as the conservative Jacobian of $\f$ with respect to $x$ (and similarly for $\J_{\f,p}(\tilde{p},\tilde{x})$). Note that $\mathcal{J}_{\f,x}(\tilde{p},\tilde{x})$ and $\mathcal{J}_{\f,p}(\tilde{p},\tilde{x})$ are obtained through the projections of the conservative Jacobian $\J_\f$ onto the $\tilde{p}$ and $\tilde{x}$ coordinates.}

Conservative Jacobians extend the concept of gradient to nonsmooth almost everywhere differentiable functions. Moreover, the conservative Jacobian $\mathcal{J}_{\f}$ coincides with the gradient $\nabla_x \f(x)$ almost everywhere \cite[Theorem 1]{bolte2021conservative}. \rz{If $\varphi$ is convex, then the standard subdifferential $\partial f$ is a conservative Jacobian for $\varphi$.}

The most useful property that conservative Jacobians possess is that they admit the chain rule \cref{eq:CJ_chain_rule}. This immediately implies the following composition rule.
\begin{lemma}[\hspace{1sp}{\cite[Lemma 6]{bolte2021conservative}}]\label{lemma:CJ_chain_rule}
Given two path-differentiable functions $\f:\Rx\to\Rp$, $\g:\Rp\to\Ru$, with conservative Jacobians $\mathcal{J}_{\f}$ and $\mathcal{J}_{\g}$, the function $\f\circ\g$ is path-differentiable with conservative Jacobian $\mathcal{J}_{\f \circ \g}(x)=\mathcal{J}_{\f}(\g(x))\mathcal{J}_{\g}(x)$.
\end{lemma}
Unfortunately, not every locally Lipschitz function is path-differentiable. For this reason, we restrict our attention to \textit{\rz{definable} functions}.
%
%
\begin{definition}[\hspace{1sp}{\cite[Definitions 1.4 and 1.5]{coste1999introduction}}]\label{def:definable}
\rz{A collection $\mathcal{O}=(\mathcal{O}_n)_{n\in\N}$, where each $\mathcal{O}_n$ contains subsets of $\R^n$, is an \textit{o-minimal structure} on $(\R,+,\cdot)$ if
\begin{enumerate}
    \item all semialgebraic subsets of $\R^n$ belong to $\mathcal{O}_n$;
    \item the elements of $\mathcal{O}_1$ are precisely the finite unions of points and intervals;
    \item $\mathcal{O}_n$ is a boolean subalgebra of the powerset of $\R^n$;
    \item if $A\in \mathcal{O}_n$ and $B\in\mathcal{O}_m$, then $A\times B \in \mathcal{O}_{n+m}$;
    \item if $A\in\mathcal{O}_{n+1}$, then the set containing the elements of $A$ projected onto their first $n$ coordinates belongs to $\mathcal{O}_n$.
\end{enumerate}
A subset of $\R^n$ which belongs to $\mathcal{O}$ is said to be \textit{definable} (in an o-minimal structure). A function $\varphi:\R^n\to\R^p$ is \textit{definable} if its graph $\{(x,v):v=\varphi(x)\}$ is definable.}
\end{definition}
\rz{Definable} functions possess the following useful property.
\begin{lemma}[\hspace{1sp}{\cite[Proposition 2]{bolte2021conservative}}]\label{lemma:CJ_semialg_implies_path_diff}
All locally Lipschitz \rz{definable} functions are path-differentiable.
\end{lemma}
Another crucial property of locally Lipschitz \rz{definable} functions is that obey a nonsmooth version of the implicit function theorem.
\begin{lemma}[\hspace{1sp}{\cite[Theorem 5]{bolte2021nonsmooth_extended}}]\label{lemma:CJ_IFT}
Let $\f:\R^{n_x}\times\R^{n_p}\to\R^{n_x}$ be a locally Lipschitz \rz{definable function} and let $\mathcal{J}_{\f}$ be its conservative Jacobian. Suppose $\f(\tilde{x},\tilde{p})=0$ for some $\tilde{x}\in\R^{n_x}$ and $\tilde{p}\in\R^{n_p}$. Assume that $\mathcal{J}_{\f}$ is convex and that for every $[U~V]\in \mathcal{J}_{\f}(\tilde{x},\tilde{p})$ the matrix $U$ is invertible. Then there exists a neighborhood $\mathcal{N}(\tilde{x})\times\mathcal{N}(\tilde{p})$ of $(\tilde{x},\tilde{p})$ and a path differentiable \rz{definable} function $x:\mathcal{N}(\tilde{p})\to\mathcal{N}(\tilde{x})$ such that for all $p\in \mathcal{N}(\tilde{p})$ it holds that $\f(x(p),p)=0$, and the conservative Jacobian $\mathcal{J}_x$ of $x$ is given for all $p\in\mathcal{N}(\tilde{p})$ by
\begin{align*}
\mathcal{J}_x(p)=\{ [-U^{-1}V] : [U~V]\in \mathcal{J}_{\f}(x(p),p) \}.
\end{align*}
\end{lemma}

Path-differentiable \rz{definable} functions can be minimized using simple projected-gradient based scheme as outlined in \cref{alg:CJ_min_of_path_diff_func}.

\begin{algorithm}
\caption{Minimization of path-differentiable functions}\label{alg:CJ_min_of_path_diff_func}
\begin{algorithmic}[1]
\Input $x^0$, $\{\alpha_k\}_{k\in\N}$.
\Init $k=0$.
\While{not converged}
\State $p^k\in \mathcal{J}_{\f}(x^k)$
\State $x^{k+1}=P_{\PP}[x^k-\alpha_kp^k]$
\State $k\gets k+1$
\EndWhile
\end{algorithmic}
\end{algorithm}

Where we used $P_\PP$ to denote the projector to the set $\PP$. Typically, we stop the algorithm e.g. when $\|p^{k}-p^{k-1}\|<\texttt{tol}$ for some positive tolerance $\texttt{tol}$, or after exceeding a certain number of iterations. The following results demonstrates that, under certain conditions on the step-size $\{\alpha_k\}_{k\in\N}$, \cref{alg:CJ_min_of_path_diff_func} is guaranteed to converge to a critical point of $\f$.
\begin{lemma}[\hspace{1sp}{\cite[Theorem 6.2]{davis2020stochastic}}]\label{lemma:CJ_stepsize_condition}
Assume that $\f$ is path-differentiable and \rz{definable}, that $\PP$ is a polytopic set, that the stepsizes $\{\alpha_k\}_{k\in\N} \subset \Rpp$ satisfy
\begin{align}\label{eq:CJ_stepsizes}
\sum_{k=0}^{\infty} \alpha_k = \infty ,~~~ \sum_{k=0}^{\infty} \alpha_k^2 < \infty,
\end{align}
and that $\sup_k \|x^k\| < \infty$. Then $x^k$ as obtained via \cref{alg:CJ_min_of_path_diff_func} converges to a critical point of $\f$, i.e., a point $\tilde{x}$ for which $0\in\mathcal{J}_{\f}(\tilde{x})$.
\end{lemma}

The assumption on the boundedness of the conservative Jacobians is not restrictive in practice and it is generally satisfied under a suitable stepsize choice, or if $\PP$ is a bounded set. Readers should refer to \cite{davis2020stochastic} for more details.
\section{Differentiating the MPC policy}\label{section:DIFF}
In this section, we rewrite \cref{eq:PF_mpc} in a more convenient form and then show that, under certain conditions, the map \mpce admits a conservative Jacobian, leading to a descent algorithm for \cref{eq:PF_problem_2}.

\subsection{Writing the MPC problem as a QP}\label{subsection:DIFF_writing_mpc_as_qp}

\rz{Problem \cref{eq:PF_mpc} can be reformulated as a quadratic program in standard form (see e.g. \cite[Section III]{wang2009fast})}
\begin{align}\label{eq:DIFF_qp}
\begin{split}
\operatorname*{minimize}_{y} & \quad \frac{1}{2}y^\top \QQ(p) y + \qq(\bar{p})^\top y\\
\mathrm{subject~to} & \quad \FF(p) y = \ff(\bar{p}),\\
& \quad \GG(p) y \leq \gg(\bar{p}),
\end{split}
\end{align}
where $\p:=(\bar{x}_t,p)$. We denote with $n_\text{in}$ and $n_\text{eq}$ the number of inequality and equality constraints in \cref{eq:DIFF_qp}, respectively. Note that the parameter $p$ can potentially affect all terms in the cost and in the constraints, whereas the initial condition $\bar{x}_t$ can only affect the linear part of the cost and the affine term in the constraints.

To simplify the computation of the conservative Jacobian, we operate on the Lagrange dual problem associated to \qp. In this case, the constraints are parameter-independent, as $\p$ only affects the cost function of the problem. To ensure that the dual problem has a unique solution for every value of $\p$, we impose the following assumption.
\begin{assumption}
\label{ass:DIFF_strong_convexity}
For all parameter vectors $\p$ in some polytopic set $\PPP$, the matrix $\QQ(p)$ in \cref{eq:DIFF_qp} is positive definite, problem \cref{eq:DIFF_qp} is feasible and satisfies the linear independence constraint qualification (LICQ).
\end{assumption}

Recall that \cref{eq:DIFF_qp} satisfies the LICQ if given an optimizer $y(\p)$, the rows of $\GG$ associated with the active inequality constraints and the rows of $\FF$ are linearly independent. Note that the LICQ assumption holds if, for example, the constraints on $x_{k|t}$ and $u_{k|t}$ in \cref{eq:PF_mpc} are simple box constraints
\begin{align*}
x_\text{min} \leq x_{k|t} \leq x_\text{max} ,~~~ u_\text{min} \leq u_{k|t} \leq u_\text{max},
\end{align*}
for some $x_\text{min},x_\text{max}\in\Rx$, $x_\text{min}<x_\text{max}$, and $u_\text{min},u_\text{max}\in\Ru$, $u_\text{min}<u_\text{max}$.

The feasibility condition in \cref{ass:DIFF_strong_convexity} can be restrictive in practical scenarios. We propose a simple extension of our method that can deal with losses of feasibility in \cref{subsection:EXT_infeasibility}.

Under \cref{ass:DIFF_strong_convexity}, we can obtain the Lagrange dual of \qp following the procedure outlined in \cref{appendix:appC}.
\begin{align}\label{eq:DIFF_dual}
\begin{split}
\operatorname*{minimize}_{z} & \quad \frac{1}{2} z^\top \HH(\p) z + \hh(\p)^\top z\\
\mathrm{subject~to} & \quad Ez \geq 0.
\end{split}
\end{align}
Note that in \dual the parameters $p$ and $\bar{x}_t$ only affect the quadratic part $\HH$ and the linear part $\hh$ of the cost, whereas the matrix $E$ in the constraints is parameter-independent.

The primal solution $y(\p)$ can be obtained from the dual solution $z(\p)=(\lambda(\p),\mu(\p))$ as
\begin{align}\label{eq:PF_primal_sol_from_dual_sol}
y(\p)=-\QQ(p)^{-1}(\FF(p)^\top \mu(\p) + \GG(p)^\top \lambda(\p)+\qq(\p)).
\end{align}

\subsection{Writing the dual as fixed point condition}
%
To obtain the conservative Jacobian $\mathcal{J}_z$ of the dual optimizer, we follow the procedure proposed in \cite{bolte2022differentiating}: we write the optimality conditions of \cref{eq:DIFF_dual} as a fixed point equation $\F(z,\p)=0$, obtain the conservative Jacobian of $\F$ with respect to $z$ and $\p$, and apply the implicit function theorem in \cref{lemma:CJ_IFT}. Since \cref{eq:DIFF_dual} is a quadratic program, a necessary and sufficient condition for optimality \cite[Theorem 3.67]{beck2017first} is
\begin{align}\label{eq:DIFF_opt_cond_1}
0 \in \HH(\p)z+\hh(\p) + N_C(z),
\end{align}
where $N_C$ is the normal cone mapping of $C:=\{z\in\R^{n_z}:Ez \geq 0\}$, with $n_z = n_\text{in} + n_\text{eq}$ \cite[Example 3.5]{beck2017first}. Leveraging \cite[Corollary 27.3]{bauschke2017convex}, we have that \cref{eq:DIFF_opt_cond_1} is equivalent to
\begin{align}\label{eq:DIFF_monotone_inclusion}
0\in P_C[ z - \gamma (\HH(\p)z+\hh(\p)) ] - z =: \F(z,\p),
\end{align}
where $\gamma\in\Rpp$ is a positive scalar and $P_C:\R^{n_z}\to C$ is the projection operator to the set $C$.
To ensure the existence of the conservative Jacobian of $\F$, we impose the following assumption.
\begin{assumption}\label{ass:DIFF_semi_alg}
The maps $\QQ(p)$, $\qq(\p)$, $\FF(p)$, $\ff(\p)$, $\GG(p)$, $\gg(\p)$ are locally Lipschitz and \rz{definable}.
\end{assumption}
\rz{\cref{ass:DIFF_semi_alg} is not restrictive in practice as definable functions include most common functions of interest in the field of optimization and control. For example, all semialgebraic functions, real analytic functions (restricted to a definable domain), and any product, sum, inversion, and composition of definable functions are definable \cite{coste1999introduction}. Moreover, derivatives of definable functions are definable \cite[Lemma 6.1]{coste1999introduction}, meaning that if $F$ and $\phi$ are obtained by linearizing $f$ (which is definable by \cref{ass:CLopt_f_is_path_diff}) using the dynamic linearization technique outlined in \cref{subsection:PF_choosing_st}, the definability assumption is immediately satisfied.}


\begin{lemma}\label{lemma:DIFF_F_is_path_diff}
Under \cref{ass:DIFF_semi_alg,ass:DIFF_strong_convexity}, $\F$ is locally Lipschitz \rz{definable}.
\end{lemma}
\begin{proof}
The projector $P_C$ is given by
\begin{align}\label{eq:DIFF_MPC_P_C}
P_C(z)\! = \! \operatorname{diag}(P_{\R_{\geq 0}}(\lambda_1),\dots,P_{\R_{\geq 0}}(\lambda_{n_\text{in}}),\mu_1,\dots,\mu_{n_\text{eq}}),
\end{align}
where $P_{\R_{\geq 0}}:\R\to\R_{\geq 0}$ is the one-dimensional projector to the set of non-negative real numbers
\begin{align}\label{eq:DIFF_MPC_P_nonneg_orth}
P_{\R_{\geq 0}}(z) = \max \{ 0,z \}.
\end{align}
The function $P_{\R_{\geq 0}}$ is locally Lipschitz and piecewise linear, therefore \rz{definable}. We conclude that $P_C$ is locally Lipschitz \rz{definable}.

Next, the function $z\mapsto z-\gamma(\HH(\p)z+\hh(\p))$ is linear in $z$, and therefore both \rz{definable} and Lipschitz. Moreover, thanks to \cref{ass:DIFF_semi_alg}, we have that both $\HH$ and $\hh$ are locally Lipschitz \rz{definable} in $\p$ since they are constructed as products or sums of locally Lipschitz \rz{definable} functions (as shown in \cref{appendix:appC}), and these operations preserve both local Lipschitz continuity and \rz{definability} \cite[Corollary 2.9]{coste2000introduction}. Note that $\QQ^{-1}(p)$ is also \rz{locally Lipschitz definable since each of its entries is the ratio of two polynomial functions (i.e. semialgebraic) of the entries of $\QQ$}. We conclude that $\p\mapsto z-\gamma(\HH(\p)z+\hh(\p))$ is both locally Lipschitz \rz{definable} in $\p$.

We conclude that $\F$ is locally Lipschitz \rz{definable} since it is the composition of locally Lipschitz \rz{definable} functions.
\end{proof}
The conservative Jacobian of the dual variable $z$ can now be readily obtained by applying the implicit function theorem in \cref{lemma:CJ_IFT} to the map $\F$.
\begin{theorem}\label{thm:main_thm}
Under \cref{ass:DIFF_semi_alg,ass:DIFF_strong_convexity}, the optimizer $z(\p)$ of \dual is unique and locally Lipschitz \rz{definable} for any $\bar{p}\in\PPP$. Its conservative Jacobian $\mathcal{J}_z(\bar{p})$ contains elements of the form $-U^{-1}V$, where
\begin{subequations}
\label{eq:DIFF_cons_jac_fixed_point}\begin{align}
U & \in J_{P_C}(I-\gamma \HH(\p)) -I, \label{eq:DIFF_cons_jac_fixed_point_1}\\
V & \in -\gamma J_{P_C}(Az+B),\label{eq:DIFF_cons_jac_fixed_point_2}
\end{align}
\end{subequations}
with $J_{P_C} \in \mathcal{J}_{P_C}(z - \gamma(\HH(\p)z+\hh(\p)))$, $A\in\mathcal{J}_{\HH}(\p)$, $B\in\mathcal{J}_{\hh}(\p)$.
\end{theorem}
\begin{proof}
See \cref{appendix:appD}.
\end{proof}
\begin{remark}
The proof in \cref{appendix:appD} also establishes that we can always choose $J_{P_C}$ of the form
\begin{align*}
J_{P_C} = \operatorname{diag}(\operatorname{sign}(\lambda_1),\dots,\operatorname{sign}(\lambda_{n_\text{in}}),1,\dots,1),
\end{align*}
justifying our choice of working with the dual problem \cref{eq:DIFF_dual} instead of the primal \cref{eq:DIFF_qp}.
\end{remark}
The conservative Jacobian $\mathcal{J}_y(\p)$ of the primal optimizer $y(\p)$ can then easily be retrieved from $\mathcal{J}_z(\p)$ using \cref{eq:PF_primal_sol_from_dual_sol}. For simplicity, define
\begin{align*}
\mathcal{G}(z,\p):=-\QQ(p)^{-1}(\FF(p)^{\top}\mu+\GG(p)^{\top}\lambda+\qq(\p)).
\end{align*}
\begin{corollary}
Under \cref{ass:DIFF_semi_alg,ass:DIFF_strong_convexity}, the optimizer $y(\p)$ of \qp is unique and locally Lipschitz \rz{definable} for any $\p\in\PPP$. Its conservative Jacobian $\mathcal{J}_y(\p)$ contains elements of the form
\begin{align}\label{eq:PF_cons_jac_of_primal}
W -\QQ(p)^{-1}[\GG(p)^{\top}~\FF(p)^{\top}]Z \in \mathcal{J}_y(\p),
\end{align}
where $Z\in\mathcal{J}_z(\p)$ and $W\in \mathcal{J}_{\mathcal{G},\p}(z,\p)$.
\end{corollary}
\begin{proof}
Follows immediately from the fact that composition preserves the local Lipschitz continuity and \rz{definability}, and by applying the chain rule of differentiation to \cref{eq:PF_primal_sol_from_dual_sol}.
\end{proof}

The algorithm below summarizes a procedure for computing the conservative Jacobians $\mathcal{J}_{\operatorname{MPC},\bar{x}_t}$ and $\mathcal{J}_{\operatorname{MPC},p}$.
\begin{algorithm}[H]
\caption{Computing $\mathcal{J}_{\operatorname{MPC}}(\p)$}\label{alg:PF_cons_jac}
\begin{algorithmic}[1]
\Input $\bar{p}$
\State Solve \cref{eq:PF_mpc} and get dual optimizers $z=(\lambda,\mu)$.
\State Compute $[U~V]\in\mathcal{J}_{\F}(z,\p)$ using \cref{eq:DIFF_cons_jac_fixed_point}.
\State Compute $Z = -U^{-1}V \in \mathcal{J}_z(\bar{p})$.
\State Compute $\mathcal{J}_y(\bar{p})$ using \cref{eq:PF_cons_jac_of_primal}.
\State \Return $\mathcal{J}_{\operatorname{MPC}}(\bar{p})$ (extracted from $\mathcal{J}_y(\bar{p})$).
\end{algorithmic}
\end{algorithm}
Notice that $\mathcal{J}_{\operatorname{MPC},\bar{x}_t}$ and $\mathcal{J}_{\operatorname{MPC},p}$ are contained in $\mathcal{J}_y$, and we can therefore retrieve them by selecting the appropriate entries in $\mathcal{J}_y(\bar{x}_t,p)$.

\rz{The procedure outlined so far allows for the computation of conservative Jacobians of quadratic programs. Extending this method to more general classes of problems is a promising direction for future research. One way to proceed could be to apply the implicit function theorem to the optimality conditions of the nonlinear problem, as done in \cite{gros2019data}. In this case, however, it's unclear whether the resulting Jacobian will be conservative.}
\section{Closed-loop optimization scheme}\label{section:CLopt}
\subsection{Backpropagation}\label{subsection:CLopt_backprop}
Next, we develop a modular mechanism, based on backpropagation, to obtain the conservative Jacobian of the entire closed-loop trajectory $\bar{x}$ using the individual conservative Jacobians of each optimization problem.

In machine learning, \textit{backpropagation} is often used to efficiently construct gradients with respect to design parameters of algorithms involving several successive steps. The idea is to compute the gradients of each step and combine them using the chain rule, eliminating redundant calculations and improving efficiency.

In our case, the closed-loop dynamics can be expressed as a recursive equation
\begin{align}
\bar{x}_{t+1} = f(\bar{x}_t, \mpce), \label{eq:CLopt_closed_loop_dynamics}
\end{align}
where every state $\bar{x}_{t+1}$ depends solely on its predecessor $\bar{x}_t$, and the design parameters $p$. To be able to propagate the conservative Jacobians through the dynamics of the system, we require $f$ to be path-differentiable.
\begin{assumption}\label{ass:CLopt_f_is_path_diff}
The function $f$ is locally Lipschitz \rz{definable}.
\end{assumption}
Under \cref{ass:CLopt_f_is_path_diff}, we can compute the conservative Jacobian $\mathcal{J}_{\bar{x}_{t+1}}(p)$ of the state $\bar{x}_{t+1}$ with respect to the design parameters $p$ recursively as follows:
\begin{align}
&\mathcal{J}_{\bar{x}_{t+1}}(p) = \mathcal{J}_{f,x}(\bar{x}_t,\bar{u}_t) \mathcal{J}_{\bar{x}_t}(p) \! + \! \mathcal{J}_{f,u}(\bar{x}_t,\bar{u}_t) \left[ \mathcal{J}_{\operatorname{MPC},\bar{x}_t}(\bar{x}_t,p) \right. \notag \\ & \hspace{1.6cm} \cdot \left. \mathcal{J}_{\bar{x}_t}(p) + \mathcal{J}_{\operatorname{MPC},p}(\bar{x}_t,p) \right]. \label{eq:PF_closed_loop_backprop_jac}
\end{align}
Note that $\mathcal{J}_{\bar{x}_{t+1}}(p)$ depends on $\mathcal{J}_{\bar{x}_t}(p)$, and since $\bar{x}_0$ is given, we have $\mathcal{J}_{\bar{x}_0}(p)=0$. As a result, we can easily construct an algorithm that computes the conservative Jacobian of the closed-loop trajectory $\bar{x}$ for a given value of $p$ iteratively. The algorithm, summarized in \cref{alg:backprop}, can be implemented online, as the closed-loop is being simulated and the values of $\bar{x}_t$ are being measured. Note that the simulation needs to span the entire horizon $T$.

\begin{algorithm}
\caption{Backpropagation}\label{alg:backprop}
\begin{algorithmic}[1]
\Input $\bar{x}_0$, $p$.
\Init $\mathcal{J}_{\bar{x}_0}(p)=0$.
\For{$t=0$ to $T$}
\State Solve \cref{eq:PF_mpc} and set $\bar{u}_t=\mpce$.
\State Get next state $\bar{x}_{t+1}=f(\bar{x}_t,\bar{u}_t)$.
\State Compute $\mathcal{J}_{\operatorname{MPC}}(\bar{x}_0,p)$ using \cref{alg:PF_cons_jac}.
\State Compute $\mathcal{J}_{\bar{x}_{t+1}}(p)$ using \cref{eq:PF_closed_loop_backprop_jac}.
\EndFor
\State \Return $\bar{x}=\bar{x}(p)$ and $\mathcal{J}_{\bar{x}}(p)$.
\end{algorithmic}
\end{algorithm}

\begin{proposition}\label{prop:PF_cons_jac_closed_loop}
Under \cref{ass:DIFF_semi_alg,ass:DIFF_strong_convexity,ass:CLopt_f_is_path_diff}, the closed-loop trajectory $\bar{x}$ is locally Lipschitz \rz{definable} in $p$, with conservative Jacobian $\mathcal{J}_{\bar{x}}(p)$ as given by \cref{alg:backprop}.
\end{proposition}
\begin{proof}
The closed-loop $\bar{x}$ is locally Lipschitz \rz{definable} since it is given by the composition of locally Lipschitz \rz{definable} functions. We now prove by induction that \cref{alg:backprop} produces $\mathcal{J}_{\bar{x}}(p)$. First, $\mathcal{J}_{\bar{x}_0}(p)=0$ since $\bar{x}_0$ is fixed \emph{a priori}. Next, suppose $\mathcal{J}_{\bar{x}_t}(p)$ has been computed correctly by the algorithm. The correctness of $\mathcal{J}_{\bar{x}_{t+1}}(p)$ follows immediately \cref{eq:PF_closed_loop_backprop_jac} and \cref{lemma:CJ_chain_rule}.
\end{proof}

\subsection{Optimization algorithm}\label{subsection:CLopt_opt}
Once the conservative Jacobian is available, we can utilize it to update the parameter $p$ with a gradient-based scheme. To guarantee convergence, it suffices to meet the conditions of \cref{alg:CJ_min_of_path_diff_func}. We, therefore, choose the following update scheme
\begin{align}\label{eq:CLopt_gd}
p^{k+1} = P_\PP [p^k - \alpha_k J],
\end{align}
for any $J = J_1 J_2$ with
\begin{align*}
J_1 \in \J_\mathcal{C}(\bar{x}^k),~~ J_2 \in \J_{\bar{x}}(p^k),
\end{align*}
where $\bar{x}^k=\bar{x}(p^k)$, and $\alpha_k$ satisfies the conditions in \cref{eq:CJ_stepsizes}. \cref{alg:PF_closed_loop} combines all the steps described so far.

\begingroup
\algrenewcommand\algorithmicindent{1.0em}
\begin{algorithm}
\caption{Closed-loop optimization scheme}\label{alg:PF_closed_loop}
\begin{algorithmic}[1]
\Input $p^0$, $\bar{x}_0$, $\{\alpha_k\}_{k\in\N}$.
\Init $k=0$.
\While{not converged}
\State Compute $\bar{x}^k=\bar{x}(p^k)$ and $\mathcal{J}_{\bar{x}}(p^k)$ with \cref{alg:backprop}.
\State Compute $J^k = J_1^k J_2^k$ with $J_1^k \in \J_\mathcal{C}(\bar{x}^k)$, $J_2^k \in \mathcal{J}_{\bar{x}}(p^k)$.
\State Update $p^{k+1}=P_\PP[p^k-\alpha_kJ]$.
\State $k\gets k+1$.
\EndWhile
\State \Return $p^*=p^k$
\end{algorithmic}
\end{algorithm}
\endgroup

As long as the map \mpc is well-defined, i.e., problem \cref{eq:PF_mpc} admits a feasible solution throughout the entirety of the execution of \cref{alg:PF_closed_loop}, we have the following.
\begin{theorem}\label{thm:CLopt_convergence}
Suppose \cref{ass:CLopt_f_is_path_diff,ass:DIFF_semi_alg,ass:DIFF_strong_convexity} hold, and that \cref{eq:PF_mpc} is feasible for all $\bar{x}_t$ and $p^k$ as setup in \cref{alg:PF_closed_loop}. Suppose $\alpha_k$ satisfies \cref{eq:CJ_stepsizes} and $\sup_k \|p^k\| < \infty$. Then $p^k$ converges to a critical point of problem \cref{eq:PF_problem}.
\end{theorem}
\begin{proof}
Since $\mathcal{J}_{\bar{x}}(p^k)$ is the conservative Jacobian of $\bar{x}$ with respect to the design parameter $p$ thanks to \cref{prop:PF_cons_jac_closed_loop}, we have from \cref{lemma:CJ_stepsize_condition} that the iterates $p^k$ are guaranteed to converge to a critical point $\bar{p}$ of the problem
\begin{align}\label{eq:CLopt_unconstrained_closed_loop}
\operatorname*{minimize}_{p\in \PP} \quad \mathcal{C}(p).
\end{align}
Moreover, since the constraints $\mathcal{H}(p)\leq 0$ are automatically satisfied if \mpc is feasible throughout the entire runtime of \cref{alg:PF_closed_loop}, we conclude that the critical point $\bar{p}$ of \cref{eq:CLopt_unconstrained_closed_loop} is also a critical point of \cref{eq:PF_closed_loop_compact_problem}, which is equivalent to \cref{eq:PF_problem}. This concludes the proof.
\end{proof}
\rz{The condition $\sup_k \|p^k\|< \infty$ holds trivially if $\PP$ is compact. Otherwise, one can augment the cost function with a regularizer that ensures boundedness of the iterates, as discussed in \cite[Section 6.1]{davis2020stochastic}.}
\rz{\begin{remark}
Since the horizon of the optimization problem \cref{eq:PF_problem} is finite, we do consider the stability of the closed-loop dynamics \cref{eq:CLopt_closed_loop_dynamics}. Indeed, our method produces MPC schemes that are optimized for a specific finite-horizon task. To obtain a controller that stabilizes \cref{eq:CLopt_closed_loop_dynamics}, a simple solution would be to use the MPC controller \cref{eq:PF_mpc} with parameter $p^*$ for $t\in\Z_{[0,T]}$, and switch to a stabilizing state-feedback controller for $t \geq T$ (e.g., an LQR). Note that if $T$ is chosen appropriately large, the closed-loop state $\bar{x}_T$ should be in a neighborhood of the origin, and a simple LQR controller (obtained by linearizing the dynamics at the origin if the system is nonlinear) should suffice.
\end{remark}}
\section{Extensions}\label{section:EXT}
\subsection{Choosing \texorpdfstring{$S_t$}{St} by linearization}\label{subsection:PF_choosing_st}

The accuracy of the approximate model $S_t$ significantly impacts the control performance. To improve precision, we can allow the nominal dynamics to vary at different time-steps within the same MPC problem
\begin{align*}
x_{k+1|t} = A_{k|t} x_{k|t} + B_{k|t} u_{k|t} + c_{k|t},
\end{align*}
and construct \rz{$S_t=\{A_{k|t},B_{k|t},c_{k|t}\}_{k=0}^{N-1}$} by linearizing $f$ along the state-input trajectory $(x_{t-1},u_{t-1})$
\begin{align*}
A_{k|t} &= \left.\pdv{f(x,u)}{x}\right\lvert_{\substack{x=x_{k+1|t-1}\\u=u_{k+1|t-1}}},~~B_{k|t} = \left.\pdv{f(x,u)}{u}\right\lvert_{\substack{x=x_{k+1|t-1}\\u=u_{k+1|t-1}}},\\
c_{k|t} &= f(x_{k+1|t-1},u_{k+1|t-1})\!-\!A_{k|t} x_{k+1|t-1} \! - \! B_{k|t} u_{k+1|t-1},
\end{align*}
with $u_{N|t-1}=u_{N-1|t-1}$. Alternatively, we can use $\bar{x}_t$ in place of $x_{1|t-1}$. If the state-input trajectory $(x_t,u_t)$ predicted by the MPC at time $t$ does not deviate significantly from $(\bar{x}_t,u_{1|t-1})$, then the linearized dynamics are expected to be a good approximation of the true system dynamics.

Since $S_t$ now depends on the entire solution $y_{t-1}:=(x_{t-1},u_{t-1})$ of the MPC problem at time $t-1$, and possibly also on $\bar{x}_t$, the computation of $\J_{\bar{x}_{t+1}}(p)$ in \cref{alg:backprop} needs to be modified:
\begin{align}
\J_{\bar{x}_{t+1}}(p) &= \J_{f,x}(\bar{x}_t,\bar{u}_t) \mathcal{J}_{\bar{x}_t}(p) \! + \! \J_{f,u}(\bar{x}_t,\bar{u}_t) \cdot \notag \\
& \quad \left[ \J_{\operatorname{MPC},\bar{x}_t}(\bar{x}_t,y_{t-1},p) \J_{\bar{x}_t}(p) \right. \notag \\
& \quad + \J_{\operatorname{MPC},y_{t-1}}(\bar{x}_t,y_{t-1},p) \J_{y_{t-1}}(p) \notag \\ & \quad \left. + \, \J_{\operatorname{MPC},p}(\bar{x}_t,y_{t-1},p)\right], \label{eq:PF_backprop_linearization}
\end{align}
where we used $\operatorname{MPC}(\bar{x}_t,y_{t-1},p)$ instead of \mpce to emphasize the dependency on $y_{t-1}$. The term $\J_{y_{t-1}}(p)$ can be constructed using a simple backpropagation rule
\begin{align}
\J_{y_{t-1}}(p) &= \J_{\text{QP},\bar{x}_{t-1}}(\bar{x}_{t-1},y_{t-2},p) \J_{x_{t-1}}(p) \notag \\ & \quad + \J_{\text{QP},y_{t-2}}(\bar{x}_{t-1},y_{t-2},p) \J_{y_{t-2}}(p) \notag \\ & \quad + \J_{\text{QP},p}(\bar{x}_{t-1},y_{t-2},p) \label{eq:PF_backprop_y_tm1}
\end{align}
where $y_{t-1} =\operatorname{QP}(\bar{x}_{t-1},y_{t-2},p)$. The modified backpropagation algorithm is given in \cref{alg:backprop_linearization}. Note that we can compute $\J_{\operatorname{MPC}}$ using \cref{alg:PF_cons_jac} by setting $\bar{p}:=(\bar{x}_t,y_{t-1},p)$.

Before beginning the simulation of the system, we need to choose the linearization trajectory $y_{-1}$ for time-step $t=0$, either heuristically, or by letting $y_{-1}$ be part of $p$, thus allowing the optimization process to select the value of $y_{-1}$ that yields the best closed-loop performance.

\begin{remark}
The linearization strategy of this section can be replaced with simpler strategies like choosing a fixed $A$ and $B$ throughout the entire MPC horizon or choosing $A_{k|t}\equiv A_{t} = \pdv{f(x,u)}{x}\lvert_{\substack{x=x_t,~~~\,\\u=u_{1|t-1}}}$ and similarly for $B$ and $c$. \rz{This however may negatively impact the performance of the MPC controller, especially when the system dynamics are highly nonlinear. The choice of $S_t$ is therefore a trade-off between computational complexity and control performance. In practice, we observed that the linearization technique of this section has an overall satisfactory performance (compare \cref{section:SIM}).}
\end{remark}

\begin{algorithm}
\caption{Backpropagation with linearization}\label{alg:backprop_linearization}
\begin{algorithmic}[1]
\Input $\bar{x}_0$, $p$, $y_{-1}$.
\Init $\mathcal{J}_{y_{-1}}(p)$, $\mathcal{J}_{\bar{x}_0}(p)=0$.
\For{$t=0$ to $T$}
\State Solve \cref{eq:PF_mpc} and set $\bar{u}_t=\operatorname{MPC}(\bar{x}_t,y_{t-1},p)$.
\State Get next state $\bar{x}_{t+1}=f(\bar{x}_t,\bar{u}_t)$.
\State Compute $\J_{\operatorname{MPC}}(\bar{x}_t,y_{t-1},p)$ using \cref{alg:PF_cons_jac}.
\State Compute $\J_{y_t}(p)$ using \cref{eq:PF_backprop_y_tm1}.
\State Compute $\J_{\bar{x}_{t+1}}(p)$ using \cref{eq:PF_backprop_linearization}.
\EndFor
\State \Return $\bar{x}=\bar{x}(p)$ and $\mathcal{J}_{\bar{x}}(p)$.
\end{algorithmic}
\end{algorithm}

\subsection{State-dependent cost and constraints}\label{subsection:PF_state_dependent_mpc}
The closed-loop performance of receding-horizon MPC schemes can be greatly improved by allowing certain elements in the MPC problem to be adapted online based on the state of the system. For example, in \cite{aboudonia2020distributed} the terminal cost and constraints are constructed online as functions of the state $\bar{x}_t$. This choice is shown to enlarge the region of attraction of the scheme.

Our backpropagation framework allows the incorporation of state-dependent elements in the MPC problem by letting $H_{x,N}$, $h_{x,N}$, and $P$ be functions of both $p$ and $\bar{x}_t$. Since both $\bar{x}_t$ and $p$ are known at runtime, the MPC problem solved online is a QP in the form
\begin{align}
\begin{split}
\operatorname*{minimize}_{y} & \quad \frac{1}{2}y^\top \QQ(\bar{x}_t,p) y + \qq(\bar{x}_t,p)^\top y\\
\mathrm{subject~to} & \quad \FF(\bar{x}_t,p) y = \ff(\bar{x}_t,p),\\
& \quad \GG(\bar{x}_t,p) y \leq \gg(\bar{x}_t,p),
\end{split}\notag
\end{align}
which differs from \qp only because $\QQ$, $\FF$, and $\GG$ now depend on both $\bar{x}_t$ and $p$. As a result, we can perform closed-loop optimization using the same algorithmic procedure as in \cref{alg:PF_closed_loop} without any modification, exception made for the symbolic expression of $\QQ$, $\FF$, and $\GG$ which depend on $\bar{p}=(\bar{x}_t,p)$ instead of only $p$.
\begin{remark}
The same procedure can be applied to the case where $H_x$, $H_u$, $h_x$, $h_u$, $Q_x$ and $R_u$ depend on $\bar{x}_t$. Moreover, one can easily incorporate cost matrices $Q_x$ and $R_u$ that also depend on $y_{t-1}$, for example by linearizing the possibly nonlinear cost function $\mathcal{C}$ along the trajectory $y_{t-1}$ and adding sufficient regularization to ensure the positive definiteness of both $Q_x$ and $R_u$. We leave such cases for future work and emphasize that our framework is flexible to tune any component of the underlying MPC problem.
\end{remark}

\subsection{Non-convex cost}\label{subsection:EXT_nonconvex_cost}
Our framework easily extends to scenarios where the quadratic cost in \cref{eq:PF_problem} is replaced with more sophisticated costs that can possibly involve other terms in addition to $\bar{x}$. Consider, for example, problem \cref{eq:PF_problem} with the cost $\mathcal{C}(\bar{x},y,z,p)$, where $y:=(y_0,\dots,y_{T})$ and $z=(z_0,\dots,z_{T})$ are the primal-dual optimizers of \qp at all time-steps. Through \cref{alg:backprop_linearization} we can include any of the optimization variables in the cost and still manage to efficiently compute the gradient of the objective by storing the conservative Jacobians $\J_{y_t}(p)$ and $\J_{z_t}(p)$ and then applying the Leibniz rule
\begin{align*}
\J_\mathcal{C}(p) = \J_{\mathcal{C},\bar{x}}(p)\J_{\bar{x}}(p) + \J_{\mathcal{C},y}(p)\J_y(p) + \J_{\mathcal{C},z}(p)\J_z(p).
\end{align*}
The conservative Jacobians $\J_y$ and $\J_z$ are already available as a by-product of \cref{alg:PF_cons_jac}.

To ensure that \cref{lemma:DIFF_F_is_path_diff} is still applicable, we only require $\mathcal{C}$ to be path-differentiable jointly in its arguments. Under this condition, the results of \cref{thm:CLopt_convergence} still hold. Note that the class of path-differentiable functions is quite large, and comprises a large selection of non-convex functions.

\subsection{Dealing with infeasibility}\label{subsection:EXT_infeasibility}
So far, we have not considered the situation where \cref{eq:PF_mpc} becomes infeasible. This can happen frequently in practice since the gradient-based optimization scheme modifies the behavior of the MPC map without guaranteeing recursive feasibility. There is, however, a simple procedure that can be used to recover from infeasible scenarios. The modification comprises two steps: first we modify \cref{eq:PF_mpc} to ensure its feasibility, then we change the cost function $\mathcal{C}$ to ensure that $p$ minimizes constraint violations.

To ensure that \cref{eq:PF_mpc} is always feasible, \rz{we introduce the slack variables $\epsilon_t$ and soften the state constraints \cite{kerrigan2000soft}} (the input constraints can always be satisfied)
\begin{align}
\operatorname*{min.}_{x_t,u_t,\epsilon_t} & \quad \mathcal{P}^{\operatorname{MPC}}_\epsilon(\epsilon_t) + \|x_{N|t}\|_P^2 + \! \sum_{k=0}^{N-1} \|x_{k|t}\|_{Q_x}^2 \! + \|u_{k|t}\|_{R_u}^2 \notag \\
\mathrm{s.t.} & \quad x_{k+1|t}=A_tx_{k|t}+B_tu_{k|t}+c_t, ~ x_{0|t} = \bar{x}_t, \notag \\
& \quad H_x x_{k|t} \leq h_x + \epsilon_{k|t} , ~ H_u u_{k|t} \leq h_u, \label{eq:EXT_mpc_feasible} \\
& \quad \epsilon_{k|t} \geq 0,~ k\in\Z_{[0,N-1]}, \notag \\
& \quad H_{x,N} x_{N|t} \leq h_{x,N} + \epsilon_{k|N}, \notag
\end{align}
To avoid unnecessary constraint violation, we penalize nonzero values of $\epsilon_t$ with the penalty function $\mathcal{P}^{\operatorname{MPC}}_\epsilon(\epsilon) = c_1 \|\epsilon\|_2^2 + c_2 \|\epsilon\|_1$, with $c_1,c_2 >0$. If $c_2$ is large enough, one can prove that $\mathcal{P}^{\operatorname{MPC}}_\epsilon$ is an exact penalty function.

\begin{lemma}[\hspace{1sp}{\cite[Theorem 1]{kerrigan2000soft}}]\label{lemma:EXT_exact_penalty_MPC}
Problem \cref{eq:EXT_mpc_feasible} has the same solution as \cref{eq:PF_mpc} as long as \cref{eq:PF_mpc} admits a solution and $c_2 > \| \lambda \|_\infty$, where $\lambda$ are the multipliers associated to the inequality constraints affecting the state in \cref{eq:PF_mpc}.
\end{lemma}
%

To ensure that, if possible, $p$ is chosen to have no constraint violations in closed-loop, we introduce a penalty function in the objective of \cref{eq:PF_problem}
\begin{align}\label{eq:EXT_opt_feasible}
\begin{split}
\operatorname*{minimize}_{p,\bar{x},y,\epsilon} & \quad \sum_{t=0}^{T} \|\bar{x}_t\|^2_{Q_x} + \mathcal{P}_\epsilon(\epsilon)\\
\mathrm{subject~to} & \quad (y_t,u_{0|t},\epsilon_t) \sim \operatorname{MPC}(\bar{x}_t,y_{t-1},p),\\
& \quad \bar{x}_{t+1}=f(\bar{x}_t,u_{0|t}), ~ \bar{x}_0,\,y_{-1}~\text{given},\\
& \quad t\in\Z_{[0,T]},
\end{split}
\end{align}
where $\epsilon:=(\epsilon_0,\dots,\epsilon_{T-1})$ and $\mathcal{P}_\epsilon(\epsilon)=c_3 \|\epsilon\|_1$ for some $c_3>0$. The introduction of $\mathcal{P}_\epsilon$ should ensure that the solution of \cref{eq:EXT_opt_feasible} satisfies $\epsilon_t=0$ for all $t\in\Z_{[0,T-1]}$. If this is the case, the optimizers $x_t$ and $u_t$ of each MPC problem \cref{eq:EXT_mpc_feasible} satisfy the nominal constraints \cref{eq:PF_constraints}, thus ensuring that $\bar{x}$ and $\bar{u}$ do too.

With some reformulation, we can equivalently write \cref{eq:EXT_opt_feasible} as
\begin{align}\label{eq:EXT_opt_feasible_2}
\underset{p}{\operatorname{minimize}} & \quad \mathcal{C}(p) + \mathcal{P}_\epsilon(\epsilon(p)),
\end{align}
where $\mathcal{C}$ and $\epsilon$ are locally Lipschitz \rz{definable} functions of $p$, and $\epsilon$ is the function that maps $p$ to the value of $\epsilon$ that solves \cref{eq:EXT_opt_feasible}. Since closed-loop constraint satisfaction is equivalent to $\epsilon(p)=0$, the goal is to obtain a solution of
\begin{align}\label{eq:EXT_opt_feasible_with_explicit_constr_2}
\begin{split}
\underset{p}{\operatorname{minimize}} & ~~ \mathcal{C}(p) \quad \text{s.t.} ~~ \epsilon(p)=0.
\end{split}
\end{align}
Under certain conditions on $\mathcal{P}_\epsilon$ and on the nature of the minimizers of \cref{eq:EXT_opt_feasible_with_explicit_constr_2}, we can prove that \cref{eq:EXT_opt_feasible_2} and \cref{eq:EXT_opt_feasible_with_explicit_constr_2} are equivalent, \rz{in which case $\mathcal{P}_\epsilon$ is an \emph{exact penalty function}}. For this, we need the following definition.
\begin{definition}
Let $p^*$ be such that $\epsilon(p^*)=0$. Problem \cref{eq:EXT_opt_feasible_with_explicit_constr_2} is \textit{calm} at $p^*$ if there exists some $\bar{\alpha}\geq 0$ and some $\epsilon>0$ such that for all $(p,u)$ with $\|p-p^*\|\leq \epsilon$ and $\epsilon(p)=u$, we have
\begin{align*}
\mathcal{C}(p)+\bar{\alpha}\|u\|\geq \mathcal{C}(p^*).
\end{align*}
\end{definition}
Calmness is a rather weak regularity condition that is verified in many situations. In finite dimensions, it holds for a dense subset of the perturbations \cite[Proposition 2.1]{burke1991exact}. For calm minimizers of \cref{eq:EXT_opt_feasible_2}, we have the following.
\begin{proposition}[{\hspace{1sp}{\cite[Theorem 2.1]{burke1991exact}}}]\label{prop:EXT_exact_penalty}
The set of local minima $p^*$ of \cref{eq:EXT_opt_feasible_with_explicit_constr_2} for which \cref{eq:EXT_opt_feasible_with_explicit_constr_2} is calm at $p^*$ coincide with the local minima of \cref{eq:EXT_opt_feasible_2} provided that the penalty parameter $c_3$ is chosen at least as large as the calmness modulus.
\end{proposition}
Generally, it may be challenging to obtain an accurate estimate of the calmness module. Nevertheless, for practical purposes, a large enough value of $c_3$ typically produces the desired effect $\epsilon=0$.

With this in mind, we can use \cref{alg:PF_closed_loop} replacing $\mathcal{C}$ with $\mathcal{C}+\mathcal{P}_\epsilon$ and the MPC problem \cref{eq:PF_mpc} with \cref{eq:EXT_mpc_feasible}. If $c_2$ and $c_3$ are sufficiently large, and under the calmness assumption of \cref{prop:EXT_exact_penalty}, we can guarantee convergence to a local minimizer of the problem with constraints \cref{eq:EXT_opt_feasible_with_explicit_constr_2}. \rz{Combining \cref{lemma:EXT_exact_penalty_MPC} and \cref{prop:EXT_exact_penalty} yields the following.}
\begin{theorem}\label{thm:EXT_rec_feas}\rz{%
Let \cref{ass:DIFF_semi_alg,ass:DIFF_strong_convexity,ass:CLopt_f_is_path_diff} hold, and let $p^*$ be the optimal parameter obtained with \cref{alg:backprop} applied to \cref{eq:EXT_opt_feasible_2}. If \cref{eq:EXT_opt_feasible_with_explicit_constr_2} is calm at $p^*$ and $c_2$ and $c_3$ are sufficiently large, then the MPC controller given in \cref{eq:PF_mpc} (without constraint relaxation) is recursively feasible for the dynamics \cref{eq:CLopt_closed_loop_dynamics}, and the closed-loop constraints \cref{eq:PF_constraints} are satisfied for all $t\in\Z_{[0,T]}$.}
\end{theorem}

\section{Simulation example}\label{section:SIM}
All simulations are done in CasADi \cite{andersson2019casadi} with the active set solver DAQP \cite{arnstrom2022dual} on a laptop with $32$ GB of RAM and an Intel(R) Core (TM) processor i7-1165G7 @ 2.80GHz. The code is available and open source\footnote{At the link \url{https://gitlab.nccr-automation.ch/rzuliani/bp_mpc_improving_closed_loop_performance_of_mpc}}. \rz{Average computation times for each BP-MPC iterations (including the simulation of $T$ time-steps and the computation of the conservative Jacobians) are 499.130 ms and 205.486} ms for the examples in \cref{subsection:SIM_cart_pend,subsec:SIM_loss_of_feasibility}, respectively.

\subsection{Input-constrained example}\label{subsection:SIM_cart_pend}
We begin by deploying our optimization scheme to solve problem \cref{eq:PF_problem_2} for the continuous time pendulum on cart system of \cite{guemghar2002predictive}
\begin{subequations}
\label{eq:SIM_ode}
\begin{align}
\ddot{x}(t) &= \frac{m\mu g \sin(\phi) - \mu \cos(\phi)(u + \mu \dot{\phi}^2\sin(\phi))}{mJ-\mu^2\cos(\phi)^2},\\
\ddot{\phi} &= \frac{J(u+\mu \dot{\phi}^2\sin(\phi))-\mu^2g\sin(\phi)\cos(\phi)}{mJ-\mu^2\cos(\phi)^2},
\end{align}
\end{subequations}
where $x$ and $\dot{x}$ are the position and velocity of the cart, respectively, and $\phi$ and $\dot{\phi}$ are the angular position and velocity of the pendulum, respectively. The goal is to steer the system to the upright equilibrium position $\bar{x}=(0,0,0,0)$ starting from $\bar{x}_0=(0,0,-\pi,0)$ (i.e., pendulum down). For the time being, we only consider the input constraints $\bar{u}(t)\in[-4,4]$ and postpone state constraints to \cref{subsec:SIM_loss_of_feasibility}. We discretize the \rz{nonlinear} ODE \cref{eq:SIM_ode} using Runge-Kutta 4 with a sampling time of $0.015$ seconds. The closed-loop objective is to minimize
\begin{align}\label{eq:SIM_cl_objective}
\mathcal{C}(\bar{x},\bar{u}) = \sum_{t=0}^{170} \|\bar{x}_t\|^2_{Q_x} + 10^{-6} \|\bar{u}_t\|^2,
\end{align}
with $Q_x=\operatorname{diag}(100,1,100,1)$. The MPC \cref{eq:PF_mpc} utilizes the same input constraints and state cost matrix $Q_x$, and an input and terminal cost parameterized as
\begin{align*}
R_u = p_0^2 + 10^{-6} ,~ P = \tilde{P}\tilde{P}^\top+ 10^{-8} I,~ \tilde{P}=\left[\begin{smallmatrix}
p_1&0&0&0\\
p_2&p_3&0&0\\
p_4&p_5&p_6&0\\
p_7&p_8&p_9&p_{10}
\end{smallmatrix}\right].
\end{align*}
We set $p=(p_0,\dots,p_{10})$, initialized with $p_0=0$ and $P$ equal to the solution of the discrete-time Algebraic Riccati equation (computed on the linearized dynamics at the origin). Note that this choice of $P$ and $R$ ensures that $P \succ 0$, and $R \succ 0$ for all $p$. We use the linearization strategy of \cref{subsection:PF_choosing_st} to obtain linear dynamics for \cref{eq:PF_mpc}. We choose a short horizon of $N=11$. In \cref{alg:PF_closed_loop} we set
\begin{align*}
\alpha_k = \frac{\rho\operatorname{log}(k+1)}{(k+1)^\eta},
\end{align*}
which fullfills the assumptions in \cref{thm:CLopt_convergence} for any $\eta\in(0.5,1]$ and $\rho>0$. Through manual tuning, we chose $\rho=5\cdot 10^{-4}$ and $\eta=0.51$.

\cref{fig:linear_iter} shows the evolution of the relative difference between the closed-loop cost attained by applying \cref{alg:PF_closed_loop} and the globally optimal cost for problem \cref{eq:PF_problem} for different numbers of iterations. Here the global optimum is obtained by solving \cref{eq:PF_problem} directly using the nonlinear solver IPOPT \cite{biegler2009large}. Note that this coincides with the solution of a nonlinear MPC problem with no terminal cost and a control horizon greater than or equal to $170$ time steps, after which the system reaches the origin. The suboptimality is negligible (less than $0.01\%$) after only a few iterations. The figure additionally shows that our method achieves a lower cost compared to an MPC with $R_u=10^{-6}$ and a fixed terminal cost derived from solving the Discrete Time Algebraic Riccati Equation (DARE), where $A$ and $B$ are chosen as the linearized dynamics at the origin.

\cref{fig:linear_time} shows the closed-loop trajectories of the horizontal and angular position, and of the input under different control policies. We note that the trajectory obtained after $10000$ iterations of \cref{alg:PF_closed_loop} effectively coincides with the optimal one, whereas the one obtained by using $p^0$ in \cref{eq:PF_mpc} differs substantially.

\rz{For this simulation example, the linearization procedure of \cref{subsection:PF_choosing_st} is crucial. Using a fixed linear model, obtained by linearizing the dynamics at the origin, results in an MPC controller that is not able to stabilize the system given the same horizon $N=11$. This highlights the importance of choosing a sufficiently accurate linear prediction model for the MPC problem. In our experience, the linearization procedure outlined in \cref{subsection:PF_choosing_st} suffices in most cases.}

\rz{In \cref{fig:nlp_no_dare}, we compare the performance of our scheme against a nonlinear MPC controller with different control horizons. The NMPC is implemented using the SQP method offered by Acados \cite{verschueren2022acados} with the solver HPIPM. At every time-step, we warm-start the next NMPC using the solution obtained in the previous time-step. At the initial time-step, the warm start trajectory is obtained by solving the NMPC problem with horizon $T$ (similar results can be obtained with shorter horizons). The initial warm starting, which we decided to add to make the comparison with our method more fair, is crucial to ensure fast convergence of the solver and to avoid numerical failures. As the horizon $N$ grows, the suboptimality of the nonlinear controller decreases; however, it never reaches the performance of our controller, which utilizes a fixed horizon $N=11$. Coincidentally, the worst-case computation time needed to solve the nonlinear MPC problem grows significantly.}

\rz{Tuning the nonlinear MPC can significantly improve its performance. This is showcased in \cref{fig:nlp}, where we added the terminal cost $x_{N|t}^\top P x_{N|t}$ (with $P$ chosen as the solution of the Algebraic Riccati Equation for the linearized dynamics at the origin) to the nonlinear MPC. In this case, the NMPC is able to outperform our scheme for horizons $N\geq 25$. For $N=25$, however, the NMPC requires about $10$ times more computation time in the worst-case scenario compared to the worst-case scenario of our scheme. If the horizon of the NMPC is chosen equal to ours (i.e., $N=11$), then our scheme attains a cost that is $10^5$ smaller compared to the NMPC.}

One might argue that BP-MPC is unnecessary when the system dynamics \cref{eq:SIM_ode} are known and noise-free. The optimal performance can more efficiently be obtained by solving a nonlinear trajectory optimization problem with a horizon larger than $170$ and applying the optimal input $u_t^{\texttt{t\_opt}}$ open loop. This choice, however, is very fragile against process noise. \cref{fig:nlp_noise} shows the closed-loop cost of our scheme (applied in receding horizon) and that of $u_t^{\texttt{t\_opt}}$ (applied in open loop), assuming that the dynamics \cref{eq:SIM_ode} are affected by a stochastic additive noise sampled uniformly from the set $\{ 0 \} \times [0,w_\text{max}] \times [0,w_\text{max}] \times \{ 0 \} $, for different values of $w_\text{max}$. Our scheme compensates for the noise, maintaining good performance.

\rz{\cref{fig:robustness} shows the percentage suboptimality (i.e., the ratio between the closed-loop cost attained by the controller, and the best achievable cost) of our controller (blue line) for $1000$ different initial conditions sampled from the set ${\bar{x}_0 + [\omega_1,\omega_2,0,0]}$, with $\omega_1,\omega_2$ sampled uniformly from the interval $[-0.02,0.02]$. We can see that the tuned MPC (with horizon $11$) performs well for about $90\%$ of all samples, but achieves a much larger cost in about $10\%$ of the cases. In addition, our scheme outperforms a nonlinear MPC with horizon $N=19$ (with terminal cost set to the solution of the DARE), and outperforms a NMPC with $N=20$ in about $90\%$ of the cases. Note however that a NMPC with horizon $N=21$ performs strictly better than our tuned controller. This indicates that the optimized MPC may not generalize well to unseen initial conditions. The question of how to robustify the tuning algorithm is a promising direction for future work.}

\begin{figure}
\centering
\input{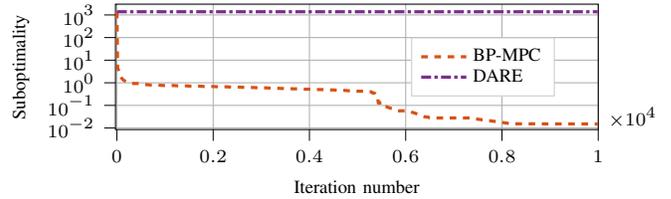}
\caption{Suboptimality between closed-loop cost and optimal cost.}\label{fig:linear_iter}
\end{figure}


\begin{figure}
\centering
\input{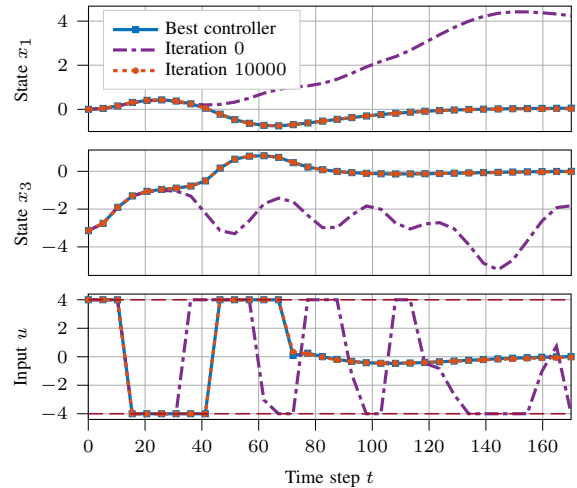}
\caption{Comparison of closed-loop state and input trajectories.}\label{fig:linear_time}
\end{figure}

\begin{figure}
\centering
\begin{tikzpicture}

\definecolor{brown1611946}{RGB}{161,19,46}
\definecolor{darkcyan0113188}{RGB}{0,113,188}
\definecolor{darkgray176}{RGB}{176,176,176}

\begin{axis}[
clip mode=individual,
height=4cm,
label style={font=\scriptsize},
legend style={font=\scriptsize},
log basis y={10},
tick align=outside,
tick label style={font=\scriptsize},
tick pos=left,
width=7.5cm,
x grid style={darkgray176},
xlabel={NLP horizon},
xmajorgrids,
xmin=11, xmax=70,
xtick style={color=black},
y grid style={darkgray176},
ylabel=\textcolor{brown1611946}{Suboptimality},
ymajorgrids,
xlabel style={yshift=8pt},
extra x ticks = {11},
extra x tick labels = {11},
extra x tick style = {yshift=0.8pt},
ymin=0.005, ymax=2000,
ymode=log,
ytick style={color=black},
ytick={0.0001,0.001,0.01,0.1,1,10,100,1000,10000,100000},
yticklabels={
  \(\displaystyle {10^{-4}}\),
  \(\displaystyle {10^{-3}}\),
  \(\displaystyle {10^{-2}}\),
  \(\displaystyle {10^{-1}}\),
  \(\displaystyle {10^{0}}\),
  \(\displaystyle {10^{1}}\),
  \(\displaystyle {10^{2}}\),
  \(\displaystyle {10^{3}}\),
  \(\displaystyle {10^{4}}\),
  \(\displaystyle {10^{5}}\)
}
]
\addplot [very thick, brown1611946, dash pattern=on 4.5pt off 1.5pt on 1.5pt off 1.5pt]
table {%
11 1143.02790031949
12 1123.04124923121
13 1012.95876206537
14 804.696306548504
15 647.877309421141
16 535.473258777775
17 431.755595160093
18 339.241533347514
19 259.931620724502
20 194.96671746536
21 143.750595431189
22 104.815309479343
23 76.078347879327
24 55.2840645845239
25 40.5786430744142
26 30.1169559369654
27 22.5788375994746
28 17.0709436269953
29 12.9968121409048
30 9.92408031950358
31 7.58865030112985
32 5.78418094845327
33 4.41450715390543
34 3.40915190624358
35 2.62931570247724
36 2.05908824460756
37 1.64456921797366
38 1.32765951065478
39 1.0679648094071
40 0.857728782068607
41 0.691020503512535
42 0.555099210386488
43 0.443411211836792
44 0.352717635293986
45 0.278830858053286
46 0.220161288943966
47 0.173661625091076
48 0.137626660429296
49 0.109535982472724
50 0.0876249796343284
51 0.0707559967685338
52 0.0581343187628857
53 0.0487079557635696
54 0.0418150056356618
55 0.0374216308182807
56 0.0348950235734925
57 0.0334373581078814
58 0.0331704246809853
59 0.0329966753975488
60 0.032745186577154
61 0.0329616671918091
62 0.0331529728971011
63 0.0332233343391798
64 0.0334861452166267
65 0.0335191183356469
66 0.0332953342683509
67 0.0325327671135041
68 0.0313456866612052
69 0.0298526752786688
70 0.0281661922196027
};
\addplot [very thick, brown1611946]
table {%
11 0.0150708437960195
12 0.0150708437960195
13 0.0150708437960195
14 0.0150708437960195
15 0.0150708437960195
16 0.0150708437960195
17 0.0150708437960195
18 0.0150708437960195
19 0.0150708437960195
20 0.0150708437960195
21 0.0150708437960195
22 0.0150708437960195
23 0.0150708437960195
24 0.0150708437960195
25 0.0150708437960195
26 0.0150708437960195
27 0.0150708437960195
28 0.0150708437960195
29 0.0150708437960195
30 0.0150708437960195
31 0.0150708437960195
32 0.0150708437960195
33 0.0150708437960195
34 0.0150708437960195
35 0.0150708437960195
36 0.0150708437960195
37 0.0150708437960195
38 0.0150708437960195
39 0.0150708437960195
40 0.0150708437960195
41 0.0150708437960195
42 0.0150708437960195
43 0.0150708437960195
44 0.0150708437960195
45 0.0150708437960195
46 0.0150708437960195
47 0.0150708437960195
48 0.0150708437960195
49 0.0150708437960195
50 0.0150708437960195
51 0.0150708437960195
52 0.0150708437960195
53 0.0150708437960195
54 0.0150708437960195
55 0.0150708437960195
56 0.0150708437960195
57 0.0150708437960195
58 0.0150708437960195
59 0.0150708437960195
60 0.0150708437960195
61 0.0150708437960195
62 0.0150708437960195
63 0.0150708437960195
64 0.0150708437960195
65 0.0150708437960195
66 0.0150708437960195
67 0.0150708437960195
68 0.0150708437960195
69 0.0150708437960195
70 0.0150708437960195
};
\end{axis}

\begin{axis}[
axis y line*=right,
clip mode=individual,
height=4cm,
label style={font=\scriptsize},
legend style={font=\scriptsize},
log basis y={10},
tick align=outside,
tick label style={font=\scriptsize},
width=7.5cm,
x grid style={darkgray176},
xmin=11, xmax=70,
xtick pos=left,
xtick style={color=black},
xtick = {},
y grid style={darkgray176},
ylabel=\textcolor{darkcyan0113188}{Computation time [ms]},
ymin=0.137007681999914, ymax=11.3777931868426,
ymode=log,
xmajorticks = false,
xticklabels={},
ytick pos=right,
ytick style={color=black},
ytick={0.01,0.1,1,10,100,1000},
yticklabel style={anchor=west},
yticklabels={
  \(\displaystyle {10^{-2}}\),
  \(\displaystyle {10^{-1}}\),
  \(\displaystyle {10^{0}}\),
  \(\displaystyle {10^{1}}\),
  \(\displaystyle {10^{2}}\),
  \(\displaystyle {10^{3}}\)
}
]
\addplot [very thick, darkcyan0113188, dash pattern=on 4.5pt off 1.5pt on 1.5pt off 1.5pt]
table {%
11 1.45435333251953
12 1.39999389648438
13 1.6169548034668
14 1.61004066467285
15 1.00564956665039
16 1.04379653930664
17 1.13558769226074
18 1.19900703430176
19 1.2214183807373
20 1.20449066162109
21 1.2664794921875
22 1.54376029968262
23 1.31702423095703
24 1.45053863525391
25 1.49035453796387
26 2.3338794708252
27 1.93238258361816
28 2.1662712097168
29 1.63793563842773
30 1.8153190612793
31 1.85823440551758
32 1.7545223236084
33 1.87110900878906
34 2.24661827087402
35 2.41279602050781
36 2.62093544006348
37 2.20680236816406
38 3.2651424407959
39 3.84306907653809
40 4.11105155944824
41 7.02929496765137
42 3.94916534423828
43 4.44293022155762
44 4.16994094848633
45 4.30464744567871
46 4.45842742919922
47 5.6304931640625
48 5.15270233154297
49 4.58264350891113
50 9.3071460723877
51 4.85682487487793
52 8.03136825561523
53 3.15403938293457
54 3.16882133483887
55 3.46970558166504
56 3.32117080688477
57 4.22191619873047
58 4.03499603271484
59 3.51572036743164
60 3.76176834106445
61 3.76677513122559
62 3.20696830749512
63 3.26728820800781
64 3.86810302734375
65 3.50427627563477
66 3.98087501525879
67 3.85141372680664
68 4.27746772766113
69 4.45771217346191
70 4.1196346282959
};
\addplot [very thick, darkcyan0113188]
table {%
11 0.167489051818848
12 0.167489051818848
13 0.167489051818848
14 0.167489051818848
15 0.167489051818848
16 0.167489051818848
17 0.167489051818848
18 0.167489051818848
19 0.167489051818848
20 0.167489051818848
21 0.167489051818848
22 0.167489051818848
23 0.167489051818848
24 0.167489051818848
25 0.167489051818848
26 0.167489051818848
27 0.167489051818848
28 0.167489051818848
29 0.167489051818848
30 0.167489051818848
31 0.167489051818848
32 0.167489051818848
33 0.167489051818848
34 0.167489051818848
35 0.167489051818848
36 0.167489051818848
37 0.167489051818848
38 0.167489051818848
39 0.167489051818848
40 0.167489051818848
41 0.167489051818848
42 0.167489051818848
43 0.167489051818848
44 0.167489051818848
45 0.167489051818848
46 0.167489051818848
47 0.167489051818848
48 0.167489051818848
49 0.167489051818848
50 0.167489051818848
51 0.167489051818848
52 0.167489051818848
53 0.167489051818848
54 0.167489051818848
55 0.167489051818848
56 0.167489051818848
57 0.167489051818848
58 0.167489051818848
59 0.167489051818848
60 0.167489051818848
61 0.167489051818848
62 0.167489051818848
63 0.167489051818848
64 0.167489051818848
65 0.167489051818848
66 0.167489051818848
67 0.167489051818848
68 0.167489051818848
69 0.167489051818848
70 0.167489051818848
};
\end{axis}

\end{tikzpicture}
\caption{Comparison of the relative suboptimality and the worst-case computation times (dashed lines) of a nonlinear MPC with different horizon lengths, and our scheme with fixed horizon $N=11$ (solid lines).}\label{fig:nlp_no_dare}
\end{figure}
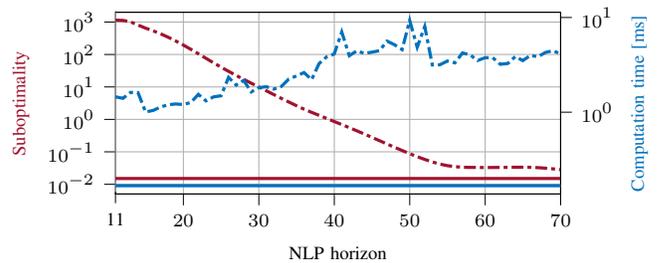

\begin{figure}
\centering
\begin{tikzpicture}

\definecolor{brown1611946}{RGB}{161,19,46}
\definecolor{darkcyan0113188}{RGB}{0,113,188}
\definecolor{darkgray176}{RGB}{176,176,176}

\begin{axis}[
clip mode=individual,
height=4cm,
label style={font=\scriptsize},
legend style={font=\scriptsize},
log basis y={10},
tick align=outside,
tick label style={font=\scriptsize},
tick pos=left,
width=7.5cm,
x grid style={darkgray176},
xlabel={NLP horizon},
xmajorgrids,
xmin=11, xmax=31,
y grid style={darkgray176},
extra x ticks = {11},
extra x tick labels = {11},
extra x tick style = {yshift=0.8pt},
xlabel style={yshift=8pt},
ylabel=\textcolor{brown1611946}{Suboptimality},
ymajorgrids,
ymin=0.005, ymax=2000,
ymode=log,
ytick style={color=black},
ytick={0.0001,0.001,0.01,0.1,1,10,100,1000,10000,100000},
yticklabels={
  \(\displaystyle {10^{-4}}\),
  \(\displaystyle {10^{-3}}\),
  \(\displaystyle {10^{-2}}\),
  \(\displaystyle {10^{-1}}\),
  \(\displaystyle {10^{0}}\),
  \(\displaystyle {10^{1}}\),
  \(\displaystyle {10^{2}}\),
  \(\displaystyle {10^{3}}\),
  \(\displaystyle {10^{4}}\),
  \(\displaystyle {10^{5}}\)
}
]
\addplot [very thick, brown1611946, dash pattern=on 4.5pt off 1.5pt on 1.5pt off 1.5pt]
table {%
11 1370.96044621936
12 1308.1949889379
13 1154.1569916677
14 796.682449898282
15 739.513944931629
16 472.583460105171
17 297.089262015888
18 135.901087274748
19 4.10500194243068
20 0.470917973707449
21 0.0381823070219779
22 0.0168737282822021
23 0.0177740826582051
24 0.0161361827549899
25 0.0148520820657621
26 0.0141058095508409
27 0.0138088149956097
28 0.0137694545633794
29 0.0137051650374013
30 0.0136912197471534
31 0.013693471155107
32 0.0136923104431268
33 0.0136894796375812
34 0.0136870826442811
35 0.0136854217100725
36 0.0136842966724952
37 0.0136835161788493
38 0.0136829616051662
39 0.0136825644017874
40 0.0136822816953111
41 0.0136820829378141
42 0.0136819421600329
43 0.0136818355019947
44 0.0136817410199081
45 0.0136816399462665
46 0.0136815182989556
47 0.013681368452824
48 0.0136811898562324
49 0.0136809889653828
50 0.0136807777016892
51 0.0136805714839298
52 0.0136803867014815
53 0.0136802383702461
54 0.0136801381403591
55 0.0136800932706552
56 0.0136801060304343
57 0.0136801740393601
58 0.0136802910475595
59 0.0136804478479767
60 0.0136806338830817
61 0.0136808378292286
62 0.0136810490021069
63 0.0136812577905326
64 0.0136814562195032
65 0.0136816381303855
66 0.0136817992129337
67 0.0136819369278213
68 0.0136820502122127
69 0.0136821392283022
70 0.0136822051328174
};
\addplot [very thick, brown1611946]
table {%
11 0.0150708437960195
12 0.0150708437960195
13 0.0150708437960195
14 0.0150708437960195
15 0.0150708437960195
16 0.0150708437960195
17 0.0150708437960195
18 0.0150708437960195
19 0.0150708437960195
20 0.0150708437960195
21 0.0150708437960195
22 0.0150708437960195
23 0.0150708437960195
24 0.0150708437960195
25 0.0150708437960195
26 0.0150708437960195
27 0.0150708437960195
28 0.0150708437960195
29 0.0150708437960195
30 0.0150708437960195
31 0.0150708437960195
32 0.0150708437960195
33 0.0150708437960195
34 0.0150708437960195
35 0.0150708437960195
36 0.0150708437960195
37 0.0150708437960195
38 0.0150708437960195
39 0.0150708437960195
40 0.0150708437960195
41 0.0150708437960195
42 0.0150708437960195
43 0.0150708437960195
44 0.0150708437960195
45 0.0150708437960195
46 0.0150708437960195
47 0.0150708437960195
48 0.0150708437960195
49 0.0150708437960195
50 0.0150708437960195
51 0.0150708437960195
52 0.0150708437960195
53 0.0150708437960195
54 0.0150708437960195
55 0.0150708437960195
56 0.0150708437960195
57 0.0150708437960195
58 0.0150708437960195
59 0.0150708437960195
60 0.0150708437960195
61 0.0150708437960195
62 0.0150708437960195
63 0.0150708437960195
64 0.0150708437960195
65 0.0150708437960195
66 0.0150708437960195
67 0.0150708437960195
68 0.0150708437960195
69 0.0150708437960195
70 0.0150708437960195
};
\draw[latex-, thick, darkcyan0113188] (axis cs:25,0.02) -- node[pos=0.5, anchor=west] {\scriptsize $\div 10$} (axis cs:25,6);
\draw[-latex, thick, brown1611946] (axis cs:11.5,0.02) -- node[pos=0.7, anchor=west] {\scriptsize $\times 10^5$} (axis cs:11.5,1000);
\end{axis}

\begin{axis}[
axis y line*=right,
clip mode=individual,
height=4cm,
label style={font=\scriptsize},
legend style={font=\scriptsize},
log basis y={10},
tick align=outside,
tick label style={font=\scriptsize},
width=7.5cm,
x grid style={darkgray176},
xmin=11, xmax=31,
xtick pos=left,
xtick style={color=black},
xmajorticks = false,
xticklabels={},
y grid style={darkgray176},
ylabel=\textcolor{darkcyan0113188}{Computation time [ms]},
ymin=0.137710660601218, ymax=10.2183913940636,
ymode=log,
ytick pos=right,
ytick style={color=black},
ytick={0.01,0.1,1,10,100,1000},
yticklabel style={anchor=west},
yticklabels={
  \(\displaystyle {10^{-2}}\),
  \(\displaystyle {10^{-1}}\),
  \(\displaystyle {10^{0}}\),
  \(\displaystyle {10^{1}}\),
  \(\displaystyle {10^{2}}\),
  \(\displaystyle {10^{3}}\)
}
]
\addplot [very thick, darkcyan0113188, dash pattern=on 4.5pt off 1.5pt on 1.5pt off 1.5pt]
table {%
11 0.875473022460938
12 0.8544921875
13 0.79035758972168
14 0.884294509887695
15 0.973939895629883
16 0.936031341552734
17 1.22570991516113
18 1.33943557739258
19 1.36804580688477
20 1.19972229003906
21 1.37519836425781
22 1.42049789428711
23 1.43098831176758
24 1.66654586791992
25 1.65843963623047
26 1.82962417602539
27 2.19202041625977
28 1.70326232910156
29 1.97219848632812
30 2.15530395507812
31 3.23295593261719
32 2.41684913635254
33 2.07972526550293
34 2.32410430908203
35 2.4569034576416
36 2.50840187072754
37 3.1282901763916
38 2.66265869140625
39 2.84290313720703
40 2.98786163330078
41 2.92325019836426
42 3.41439247131348
43 3.39174270629883
44 3.07035446166992
45 3.56149673461914
46 3.32736968994141
47 3.40437889099121
48 3.79323959350586
49 3.81731986999512
50 4.12869453430176
51 6.56771659851074
52 6.17551803588867
53 7.16447830200195
54 7.00664520263672
55 6.93178176879883
56 6.28995895385742
57 8.40163230895996
58 6.75010681152344
59 6.64305686950684
60 6.61754608154297
61 6.60157203674316
62 6.7901611328125
63 6.59370422363281
64 6.8361759185791
65 7.06887245178223
66 5.04946708679199
67 4.75358963012695
68 4.52947616577148
69 4.91642951965332
70 5.23853302001953
};
\addplot [very thick, darkcyan0113188]
table {%
11 0.167489051818848
12 0.167489051818848
13 0.167489051818848
14 0.167489051818848
15 0.167489051818848
16 0.167489051818848
17 0.167489051818848
18 0.167489051818848
19 0.167489051818848
20 0.167489051818848
21 0.167489051818848
22 0.167489051818848
23 0.167489051818848
24 0.167489051818848
25 0.167489051818848
26 0.167489051818848
27 0.167489051818848
28 0.167489051818848
29 0.167489051818848
30 0.167489051818848
31 0.167489051818848
32 0.167489051818848
33 0.167489051818848
34 0.167489051818848
35 0.167489051818848
36 0.167489051818848
37 0.167489051818848
38 0.167489051818848
39 0.167489051818848
40 0.167489051818848
41 0.167489051818848
42 0.167489051818848
43 0.167489051818848
44 0.167489051818848
45 0.167489051818848
46 0.167489051818848
47 0.167489051818848
48 0.167489051818848
49 0.167489051818848
50 0.167489051818848
51 0.167489051818848
52 0.167489051818848
53 0.167489051818848
54 0.167489051818848
55 0.167489051818848
56 0.167489051818848
57 0.167489051818848
58 0.167489051818848
59 0.167489051818848
60 0.167489051818848
61 0.167489051818848
62 0.167489051818848
63 0.167489051818848
64 0.167489051818848
65 0.167489051818848
66 0.167489051818848
67 0.167489051818848
68 0.167489051818848
69 0.167489051818848
70 0.167489051818848
};
\end{axis}

\end{tikzpicture}
\caption{Comparison of the relative suboptimality and the worst-case computation times of a nonlinear MPC with terminal cost and different horizon lengths (dashed lines), and our scheme with fixed horizon $N=11$ (solid lines).}\label{fig:nlp}
\end{figure}
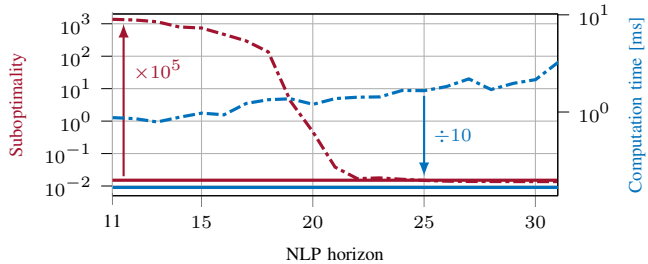

\begin{figure}
\centering
\begin{tikzpicture}

\definecolor{brown1611946}{RGB}{161,19,46}
\definecolor{darkcyan0113188}{RGB}{0,113,188}
\definecolor{darkgray176}{RGB}{176,176,176}
\definecolor{lightgray204}{RGB}{204,204,204}

\begin{axis}[
clip mode=individual,
height=3.25cm,
label style={font=\scriptsize},
legend cell align={left},
legend style={
  fill opacity=0.8,
  draw opacity=1,
  text opacity=1,
  at={(0.03,0.97)},
  anchor=north west,
  draw=lightgray204,
  row sep = -2.5pt
},
legend style={font=\scriptsize},
log basis x={10},
tick align=outside,
tick label style={font=\scriptsize},
tick pos=left,
width=8cm,
x grid style={darkgray176},
xlabel={Noise magnitude},
xmajorgrids,
xmin=9.99e-07, xmax=0.0001,
xmode=log,
xtick style={color=black},
xtick={1e-08,1e-07,1e-06,1e-05,0.0001,0.001,0.01},
xticklabels={
  \(\displaystyle {10^{-8}}\),
  \(\displaystyle {10^{-7}}\),
  \(\displaystyle {10^{-6}}\),
  \(\displaystyle {10^{-5}}\),
  \(\displaystyle {10^{-4}}\),
  \(\displaystyle {10^{-3}}\),
  \(\displaystyle {10^{-2}}\)
},
y grid style={darkgray176},
ylabel={Closed-loop cost},
ymajorgrids,
ymin=0, ymax=350000,
ytick style={color=black},
tick scale binop = \times,
y tick scale label style={at={(axis description cs:0,1.2)}, anchor=west},
]
\addplot [very thick, darkcyan0113188]
table {%
1e-06 17147.1189831241
3.53846153846154e-06 17145.0907451918
6.07692307692308e-06 17143.1995797194
8.61538461538462e-06 17141.3162968271
1.11538461538462e-05 17139.190924268
1.36923076923077e-05 17137.1969947332
1.62307692307692e-05 17135.4871757794
1.87692307692308e-05 17133.6326316857
2.13076923076923e-05 17131.3662890244
2.38461538461538e-05 17129.266848638
2.63846153846154e-05 17127.4492434547
2.89230769230769e-05 17125.0724623527
3.14615384615385e-05 17123.8090115175
3.4e-05 17121.648248329
3.65384615384615e-05 17119.6723303782
3.90769230769231e-05 17117.5534939346
4.16153846153846e-05 17114.9303003818
4.41538461538462e-05 17113.1636380325
4.66923076923077e-05 17111.7199054227
4.92307692307692e-05 17109.3772374659
5.17692307692308e-05 17107.2717010205
5.43076923076923e-05 17105.3757682277
5.68461538461538e-05 17104.7781764564
5.93846153846154e-05 17101.9980603934
6.19230769230769e-05 17097.7646705754
6.44615384615385e-05 17096.7184883051
6.7e-05 17093.8662792277
6.95384615384615e-05 17093.7629220625
7.20769230769231e-05 17089.1696632388
7.46153846153846e-05 17089.9618121146
7.71538461538462e-05 17086.5179807036
7.96923076923077e-05 17087.5075004397
8.22307692307692e-05 17083.1904679973
8.47692307692308e-05 17081.0682951971
8.73076923076923e-05 17080.0380706313
8.98461538461538e-05 17077.3380130656
9.23846153846154e-05 17075.8391088813
9.49230769230769e-05 17074.513162592
9.74615384615385e-05 17072.9902238114
0.0001 17070.4864623971
};
\addlegendentry{BP-MPC}
\addplot [very thick, brown1611946, dash pattern=on 4.5pt off 1.5pt on 1.5pt off 1.5pt]
table {%
1e-06 124464.653642965
3.53846153846154e-06 138361.968222062
6.07692307692308e-06 150108.419946674
8.61538461538462e-06 158468.916152035
1.11538461538462e-05 162989.526097872
1.36923076923077e-05 166759.996416106
1.62307692307692e-05 168710.85776104
1.87692307692308e-05 170471.766320651
2.13076923076923e-05 171702.337543821
2.38461538461538e-05 171983.105585361
2.63846153846154e-05 172169.949428127
2.89230769230769e-05 172178.632199954
3.14615384615385e-05 172180.53401569
3.4e-05 172191.504588332
3.65384615384615e-05 172309.450817278
3.90769230769231e-05 172430.852518599
4.16153846153846e-05 172784.713858317
4.41538461538462e-05 173905.964756271
4.66923076923077e-05 174423.594856939
4.92307692307692e-05 175194.145274503
5.17692307692308e-05 176304.153325717
5.43076923076923e-05 178569.700146504
5.68461538461538e-05 178663.582365407
5.93846153846154e-05 182653.666938271
6.19230769230769e-05 185240.16316663
6.44615384615385e-05 188495.998572743
6.7e-05 189669.100471754
6.95384615384615e-05 191893.10449653
7.20769230769231e-05 192175.863644068
7.46153846153846e-05 194068.751277983
7.71538461538462e-05 194512.53254229
7.96923076923077e-05 192271.645522379
8.22307692307692e-05 196008.915290045
8.47692307692308e-05 195633.535496869
8.73076923076923e-05 198221.082749375
8.98461538461538e-05 939878.171979803
9.23846153846154e-05 1215978.26430356
9.49230769230769e-05 1254605.6841986
9.74615384615385e-05 1768118.93596987
0.0001 1938268.65881523
};
\addlegendentry{Trajectory Optimization}
\end{axis}

\end{tikzpicture}
\caption{Closed-loop median performance (over $1000$ random noise samples) of nonlinear trajectory optimization (feedforward) and BP-MPC (receding horizon) with different noise magnitudes.}\label{fig:nlp_noise}
\end{figure}
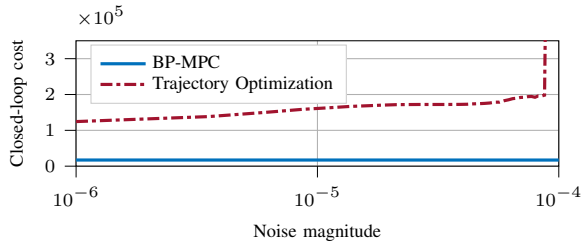

%
\begin{figure}
\centering
\input{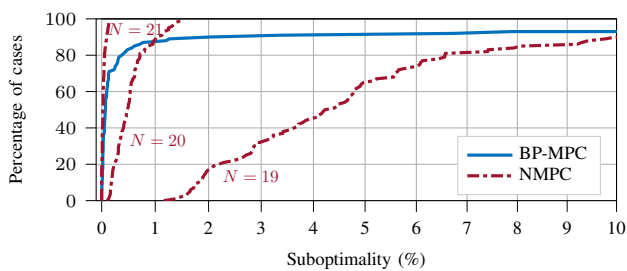}
\caption{Relative suboptimality of the tuned MPC scheme (with horizon $11$) and nonlinear MPC schemes with variable horizon for a set of $1000$ initial conditions.}\label{fig:robustness}
\end{figure}

\subsection{Input and state-constrained example}\label{subsec:SIM_loss_of_feasibility}
\rz{In this section we consider the same dynamics \cref{eq:SIM_ode}, discretized using RK4 with a sampling time of $0.05$, but this time with initial condition $x(0)=(-3,0,0,0)$ (i.e., pendulum up, cart at $-3$ meters from the origin). In this case, the goal is to steer the system to the origin while mainting the pendulum close to the upright position as not to exceed the constraints $\dot{x}(t)\in[-0.6,0.6]$, $\phi(t)\in[-0.1,0.1]$ and $\dot{\phi}(t)\in[-0.6,0.6]$. The input constraints are $u(t)\in[-0.9,0.9]$. We choose the MPC horizon as $N=6$, with $T=120$, and closed-loop cost $Q_x=\operatorname*{diag}(1,0.01,1,0.1)$, $R_u=0.01$. Moreover, we use the same choice of $p$ as in \cref{subsection:SIM_cart_pend} with the same initialization.}

\rz{This new control task is very challenging from a safety perspective, as the controller needs to reduce the aggressiveness of the control action to avoid violating the tight constraints. This is similar to the problem of controlling the position of a segway without falling.}
%
We therefore utilize the soft-constrained MPC described in \cref{eq:EXT_mpc_feasible} with penalty parameters $c_1=15$ and $c_2=15$, and apply the optimization scheme with the objective function in \cref{eq:EXT_opt_feasible_2}, with $P_\epsilon(\epsilon)=60\mathbf{1}^\top \epsilon + 40 \epsilon^\top \epsilon$ (where $\epsilon:=(\epsilon_1,\epsilon_2,\dots,\epsilon_T)$ contains the slack variables of all the optimization problems, each of which spans $N$ time-steps, and $\mathbf{1}$ is the vector of all ones). The results can be seen in \cref{fig:slack_state}, where the constraints are satisfied and $\epsilon=0$. The effect of a penalty on the constraint violation induces the optimization algorithm to favor values of $p$ that maintain small constraint violations. This does not happen if $c_3=c_4=0$, as evidenced by \cref{fig:slack_state} (\rz{note that in this case we still penalize the slacks within each MPC problem, i.e., $c_1=c_2=15$}). The closed-loop cost and constraint violations of the BP-MPC and the MPC with fixed terminal cost (equal to the solution of the DARE) are summarized in \cref{table:cost_and_violation}.
\begin{table}[h]
\centering
\begin{tabular}{c c c }
\toprule
& \textbf{Closed-loop cost} & \textbf{Constraint violation} \\
\midrule
DARE & $377.391$ & $0.713$ \\
BP-MPC (penalty) & $390.824$ & $0$ \\
BP-MPC (no penalty) & $180.979$ & $43.829$ \\
Best achievable & $380.244$ & $0$ \\
\bottomrule
\end{tabular}
\caption{Comparison of closed-loop cost and constraint violation.}\label{table:cost_and_violation}
\end{table}
\begin{figure}
\centering
\input{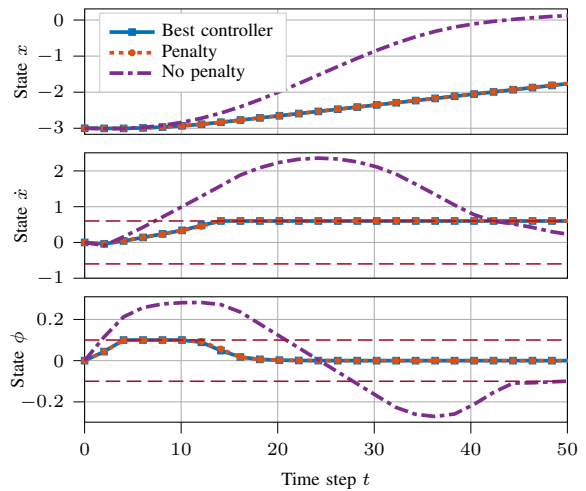}
\caption{Comparison of closed-loop state trajectories {\sffamily\normalfont $x(t)$, $\dot{x}(t)$, and $\dot{\phi}(t)$} under different control policies.}\label{fig:slack_state}
\end{figure}

\subsection{Comparison with \mytcite{gros2019data}}\label{subsection:SIM_gros}

\rz{For completeness, we compare our method with \cite{gros2019data} on the example presented in \cref{subsection:SIM_cart_pend}. When using the method in \cite{gros2019data}, we use the same cost parameterization we used with our method, but employ a fully nonlinear MPC. Using open source code\footnote{\url{https://github.com/FilippoAiraldi/mpc-reinforcement-learning}}, we test both a quasi-Newton and a gradient-based update. However, despite some manual tuning of the stepsizes, both methods fail to achieve a performance comparable to the one of our method (the best observed cost is 37404.04, compared to 17147.9 of our method).}

\rz{The situation changes if we consider the example in \cref{subsec:SIM_loss_of_feasibility}. To make the comparison fair, we use a fixed linear prediction model for both methods (obtained by linearizing the dynamics at the origin). In this case, the method of \cite{gros2019data} converges rapidly to the optimal performance (same as our method) in a smaller number of iterations than our method. This is due to the fundamental difference in the two update schemes: \cite{gros2019data} is using a quasi-Newton scheme and it is performing one update every time-step within each iteration, requiring $1$ iteration to converge (hence $120$ parameter updates are necessary to reach optimal performance); our scheme, on the other hand, uses a gradient-based scheme and only updates once per iteration, requiring $3$ iterations in total.}

\rz{This comparison empirically suggests that, in scenarios where the dynamics are highly nonlinear (e.g., the example of \cref{subsection:SIM_cart_pend}), our method can outperform the method of \cite{gros2019data}, achieving optimal performance despite using a simpler MPC architecture. However, in scenarios where the dynamics are linear or mildly nonlinear (e.g., the example of \cref{subsec:SIM_loss_of_feasibility}), the method of \cite{gros2019data} can achieve similar performance with fewer iterations, thanks to the quasi-Newton update scheme and the possibility of updating the parameters at each time-step of every iteration. It is important to note that, additionally, our method has convergence guarantees.}
\section{Conclusion}
We proposed a backpropagation algorithm to optimally design an MPC scheme to maximize closed-loop performance. The cost and the constraints in the MPC can depend on the current state of the system, as well as on past solutions of previous MPC problems. We employed conservative Jacobians to compute the sensitivity of the closed-loop trajectory with respect to variations of the design parameter. Leveraging a non-smooth version of the implicit function theorem, we derived sufficient conditions under which the gradient-based optimization procedure converges to a critical point of the problem. Further, we extended our framework to cases where the MPC problem becomes infeasible, using nonsmooth penalty functions and derived conditions under which the closed-loop is guaranteed to converge to a safe solution.

Current work focuses on deploying our optimization scheme on more realistic real-life examples and extending it to scenarios where the system dynamics are only partially known and / or affected by stochastic noise.

\rz{Future work will also investigate the use of more advanced optimization techniques, such as second-order or accelerated methods, to improve the convergence rate of the optimization procedure and to obtain better convergence guarantees (e.g. to a local or global minimizer).}
\appendix
\subsection{Obtaining \texorpdfstring{\cref{eq:PF_closed_loop_compact_problem}}{(\ref*{eq:PF_closed_loop_compact_problem})} from \texorpdfstring{\cref{eq:PF_problem_2}}{(\ref*{eq:PF_problem_2})}}\label{appendix:appA}
Let $\bar{f}:\Rx \times \Rp \to \Rx$ be defined as $\bar{f}(\bar{x}_t,p)=f(\bar{x}_t,\mathrm{MPC}(\bar{x}_t,p))$ and $\bar{f}^n:= \bar{f}\circ \bar{f}^{n-1}$, with $\bar{f}^1:=\bar{f}$. Then the entire closed-loop state trajectory $\{\bar{x}_t\}_{t=0}^T$ can be written for all $t\in\Z_{[1,T]}$ as $\bar{x}_t=\bar{f}^t(\bar{x}_0,p)$. We define, with a slight abuse of notation, $\bar{x}: \R^{n_p} \to \R^{(T+1)n_x}$ as $\bar{x}(p):=(\bar{x}_0,\bar{f}(\bar{x}_0,p),\bar{f}^2(\bar{x}_0,p),\dots,\bar{f}^T(\bar{x}_0,p))$, where we omitted the dependency of $\bar{x}$ on $\bar{x}_0$ for simplicity. We can write the constraints in \cref{eq:PF_problem_2} compactly as $\mathcal{H}(p)\leq 0$, where $\mathcal{H}:\R^p \to \R^{(T+1)n_x}$ is defined as $\mathcal{H}(p):=\bar{H}_x \bar{x}(p) -\bar{h}_x$, with $\bar{H}_x:=\operatorname{diag}(H_x,\dots,H_x)$ and $\bar{h}_x:=(h_x,\dots,h_x)$.

\subsection{Deriving the Dual of the QP \texorpdfstring{\cref{eq:DIFF_qp}}{(\ref*{eq:DIFF_qp})}}\label{appendix:appC}

In this section, we derive the dual problem \dual of the quadratic program \qp. To do so, consider the Lagrangian function $\L$ of \qp, given by
\begin{align*}
\L(y,\lambda,\mu) = \frac{1}{2}y^\top \QQ y +\qq^\top y + \mu^\top (\FF y-\ff) + \lambda^\top (\GG y-\gg),
\end{align*}
where $\lambda \geq 0$ and $\mu$ are the Lagrange multipliers of \qp. The Lagrange dual function $d$ can be obtained by minimizing $\L$ with respect to its first argument:
\begin{align*}
d(\lambda,\mu) = \inf_{y} \frac{1}{2}y^\top \QQ y + \qq^\top y + \mu^\top (\FF y-\ff) + \lambda^\top (\GG y-\gg).
\end{align*}
Since $\L$ is a convex function of $x$, we can find its infimum by setting $\nabla_y \L(y,\lambda,\mu)$ to zero and solving for $y$:
\begin{align}
\label{eq:APP_dual_kkt1}
\nabla_y \L(y,\lambda,\mu) = \QQ y+\qq+\FF^\top \mu + \GG^\top \lambda = 0.
\end{align}
Note that since $\QQ$ is positive definite, it is invertible and its inverse is also positive definite. Therefore, the only value of $y$ solving \cref{eq:APP_dual_kkt1} is given by
\begin{align}
\label{eq:APP_dual_primal_opt}
y=-\QQ^{-1}(\FF^\top \mu + \GG^\top \lambda + \qq).
\end{align}
Substituting this value of $y$ in $\L$ we get
\begin{align*}
d(\lambda,\mu)=d(z)=-\frac{1}{2}z^\top \HH z - \hh^\top z - c,
\end{align*}
where $z=(\lambda,\mu)$, $c:= -1/2 \qq^\top \QQ^{-1} \qq$ is a constant independent of $z$, and
\begin{align}\label{eq:APP_dual_matrices}
\HH:=\begin{bmatrix}
\GG\QQ^{-1}\GG^\top & \GG\QQ^{-1}\FF^\top \\ \FF\QQ^{-1}\GG^\top & \FF\QQ^{-1}\FF^\top
\end{bmatrix} ,~
\hh := \begin{bmatrix} \GG\QQ^{-1}\qq+\gg \\ \FF\QQ^{-1}\qq+\ff \end{bmatrix}.
\end{align}
The Lagrange dual problem aims at maximizing $d$ over $z$ with the constraint $\lambda \geq 0$. Equivalently, we can minimize $-d$ subject to the same constraint:
\begin{align*}
\operatorname*{minimize}_{z} & \quad \frac{1}{2}z^\top \HH z + \hh^\top z + c\\
\text{subject to}& \quad Ez \geq 0,
\end{align*}
where \rz{$E=[I~0]$} is the matrix satisfying $Ez=\lambda$.

\begin{figure*}[!t]
\normalsize
\setcounter{MYtempeqncnt}{\value{equation}}
\setcounter{equation}{43}
\begin{align}\label{eq:APP_big_mat_1}
\begin{bmatrix}
0 &&&\\
&J_\text{wa}&&\\
&&I&\\
&&&I
\end{bmatrix}\left[
I
-\gamma\begin{bmatrix}
\GG_{\text{i}} \QQ^{-1} \GG_\text{i}^\top & \GG_\text{i} \QQ^{-1} \GG_\text{wa}^\top & \GG_\text{i} \QQ^{-1} \GG_\text{a} & \GG_\text{i} \QQ^{-1} \FF^\top \\
\GG_{\text{wa}} \QQ^{-1} \GG_\text{i}^\top & \GG_\text{wa} \QQ^{-1} \GG_\text{wa}^\top & \GG_\text{wa} \QQ^{-1} \GG_\text{a} & \GG_\text{wa} \QQ^{-1} \FF^\top \\
\GG_{\text{a}} \QQ^{-1} \GG_\text{i}^\top & \GG_\text{a} \QQ^{-1} \GG_\text{wa}^\top & \GG_\text{a} \QQ^{-1} \GG_\text{a} & \GG_\text{a} \QQ^{-1} \FF^\top \\
\FF \QQ^{-1} \GG_\text{i}^\top & \FF \QQ^{-1} \GG_\text{wa}^\top & \FF \QQ^{-1} \GG_\text{a} & \FF \QQ^{-1} \FF^\top
\end{bmatrix}\right]
\begin{bmatrix}
v_\text{i}\\v_\text{wa}\\v_\text{a}\\v_\text{eq}
\end{bmatrix}=
\begin{bmatrix}
v_\text{i}\\v_\text{wa}\\v_\text{a}\\v_\text{eq}
\end{bmatrix}
\end{align}
\begin{align}\label{eq:APP_big_mat_2}
\begin{bmatrix}
J_\text{wa}&&\\
&I&\\
&&I
\end{bmatrix}\left[
I
-\gamma\begin{bmatrix}
\GG_\text{wa} \QQ^{-1} \GG_\text{wa}^\top & \GG_\text{wa} \QQ^{-1} \GG_\text{a} & \GG_\text{wa} \QQ^{-1} \FF^\top \\
\GG_\text{a} \QQ^{-1} \GG_\text{wa}^\top & \GG_\text{a} \QQ^{-1} \GG_\text{a} & \GG_\text{a} \QQ^{-1} \FF^\top \\
\FF \QQ^{-1} \GG_\text{wa}^\top & \FF \QQ^{-1} \GG_\text{a} & \FF \QQ^{-1} \FF^\top
\end{bmatrix}\right]
\begin{bmatrix}
v_\text{wa}\\v_\text{a}\\v_\text{eq}
\end{bmatrix}=
\begin{bmatrix}
v_\text{wa}\\v_\text{a}\\v_\text{eq}
\end{bmatrix}
\end{align}
\begin{align}\label{eq:APP_big_mat_3}
\begin{bmatrix}
J_\text{wa}-I&&\\
&0&\\
&&0
\end{bmatrix}
\begin{bmatrix}
v_\text{wa}\\v_\text{a}\\v_\text{eq}
\end{bmatrix}-
\begin{bmatrix}
J_\text{wa}&&\\
&I&\\
&&I
\end{bmatrix}
\begin{bmatrix}
\tilde{\GG}_\text{wa} \tilde{\GG}_\text{wa}^\top & \tilde{\GG}_\text{wa} \tilde{\GG}_\text{a} & \tilde{\GG}_\text{wa} \tilde{\FF}^\top \\
\tilde{\GG}_\text{a} \tilde{\GG}_\text{wa}^\top & \tilde{\GG}_\text{a} \tilde{\GG}_\text{a} & \tilde{\GG}_\text{a} \tilde{\FF}^\top \\
\tilde{\FF} \tilde{\GG}_\text{wa}^\top & \tilde{\FF} \tilde{\GG}_\text{a} & \tilde{\FF} \tilde{\FF}^\top
\end{bmatrix}
\begin{bmatrix}
v_\text{wa}\\v_\text{a}\\v_\text{eq}
\end{bmatrix}=\begin{bmatrix}
0\\0\\0
\end{bmatrix}
\end{align}
\begin{align}\label{eq:APP_big_mat_4}
\underbrace{\begin{bmatrix}
I-J_\text{wa}^{-1}&&\\
&0&\\
&&0
\end{bmatrix}}_{:=\mathcal{A}_1}
\begin{bmatrix}
v_\text{wa}\\v_\text{a}\\v_\text{eq}
\end{bmatrix}-
\underbrace{\begin{bmatrix}
\tilde{\GG}_\text{wa} \tilde{\GG}_\text{wa}^\top & \tilde{\GG}_\text{wa} \tilde{\GG}_\text{a} & \tilde{\GG}_\text{wa} \tilde{\FF}^\top \\
\tilde{\GG}_\text{a} \tilde{\GG}_\text{wa}^\top & \tilde{\GG}_\text{a} \tilde{\GG}_\text{a} & \tilde{\GG}_\text{a} \tilde{\FF}^\top \\
\tilde{\FF} \tilde{\GG}_\text{wa}^\top & \tilde{\FF} \tilde{\GG}_\text{a} & \tilde{\FF} \tilde{\FF}^\top
\end{bmatrix}}_{:=\mathcal{A}_2}
\begin{bmatrix}
v_\text{wa}\\v_\text{a}\\v_\text{eq}
\end{bmatrix}=\begin{bmatrix}
0\\0\\0
\end{bmatrix}
\end{align}
\hrulefill
\end{figure*}
\setcounter{equation}{\value{MYtempeqncnt}}

\subsection{Proof of \texorpdfstring{\cref{thm:main_thm}}{Theorem \ref*{thm:main_thm}}} \label{appendix:appD}

By \cite[Theorem 1]{wachsmuth2013licq}, if LICQ holds for a convex optimization problem like \qp, then the set of multipliers is a singleton. Hence $z(p)$ is unique since it is a multiplier for \qp.

We first show that the expressions in \cref{eq:DIFF_cons_jac_fixed_point} represent the elements of the conservative Jacobian $\J_{\mathcal{F}}(z,\p)$. Note that
\begin{align}\label{eq:APP_cons_jac_1}
\J_{\mathcal{F},z}(z,\p) &= J_{P_C} \J_{[z-\gamma(\HH(\p)z+\hh(\p))],z}(z,\p) - I \\
&=J_{P_C} (I- \gamma \HH(\p) ) - I,\label{eq:APP_cons_jac_2}
\end{align}
where $C=\{ z :Ez \geq 0 \} = \R_{\geq 0}^{n_\text{in}} \times \R^{n_\text{eq}}$, and $J_{P_C}\in \J_{P_C}(\bar{z})$ with
\begin{align}\label{eq:APP_lam_bar}
\bar{z}=(\bar{\lambda},\bar{\mu}):=z-\gamma(\HH(\p)z+\hh(\p)),
\end{align}
and we have used the fact that $z-\gamma(\HH(\p)z-\hh(\p))$ is differentiable in $z$ with Jacobian equal to $I-\gamma \HH(\p)$. Next, we have
\begin{align*}
\J_{\mathcal{F},\p}(z,\p) &= J_{P_C} \J_{[z-\gamma(\HH(\p)z+\hh(\p))],\p}(z,\p) \\
&= - \gamma J_{P_C} (Az + B),
\end{align*}
where $A\in\J_{\HH}(\p)$, $B\in\J_{\hh}(\p)$, and we have used the fact that $\HH$ and $\hh$ are path-differentiable as a by-product of \cref{lemma:DIFF_F_is_path_diff}. This proves that \cref{eq:DIFF_cons_jac_fixed_point_1} and \cref{eq:DIFF_cons_jac_fixed_point_2} are the elements of the conservative Jacobian $\J_{\mathcal{F}}(z,\p)$. 

The conservative Jacobian of $P_C$ can be easily computed by recognizing from \cref{eq:DIFF_MPC_P_C} that $P_C$ is composed by a number of scalar projectors to the non-negative orthant, and a number of scalar identity functions. The identity functions have unitary Jacobian, whereas the projectors, whose explicit representation is given in \cref{eq:DIFF_MPC_P_nonneg_orth}, are convex and path differentiable (since they are Lipschitz and convex). Therefore, from \cite[Corollary 2]{bolte2021conservative}, we conclude that their subdifferential (in the sense of convex analysis) is a conservative field. This means that we can choose for each $i\in\Z_{[1,n_\text{in}]}$
\begin{align*}
\mathcal{J}_{P_{\R_{\geq 0}}}(\lambda_i)=\begin{cases}
0 & \text{if } \lambda_i < 0,\\
1 & \text{if } \lambda_i > 0,\\
[0,1] & \text{if } \lambda_i =0.
\end{cases}
\end{align*}
We can then retrieve the conservative Jacobian of $P_C$ by combining the conservative Jacobians of all the projectors
\begin{align*}
\mathcal{J}_{P_C}(\bar{z})= \operatorname{diag}(\mathcal{J}_{P_{\R_{\geq 0}}}(\bar{\lambda}_1),\dots,\mathcal{J}_{P_{\R_{\geq 0}}}(\bar{\lambda}_{n_\text{in}}),1,\dots,1).
\end{align*}
Assuming, for the time being, that every element of $\mathcal{J}_{\mathcal{F},z}(z(\p),\p)$ is invertible, where $z(\p)$ is the solution of \dual, we immediately have that $z$ is path-differentiable with conservative Jacobian $\J_z(\p)$, and $-U^{-1}V\in\J_z(\p)$ for every $[U~V]\in \J_{\mathcal{F}}(z,\p)$ thanks to \cref{lemma:CJ_IFT}.
We now only need to prove that for every $[U~V]\in \J_{\mathcal{F}}(z,\p)$, $U$ is invertible.

Let $\mathcal{I} \subset \Z_{[1,n_\text{in}]}$ denote the active inequality constraints. Remembering that $\bar{\lambda}_i$ is defined in \cref{eq:APP_lam_bar}, and that $\lambda_i=P_{\R_{\geq 0}}[\bar{\lambda}_i]$, we now prove the following:
\begin{enumerate}
    \item $\bar{\lambda}_i>0$ if $i\in\mathcal{I}$ with $\lambda_i>0$ (active constraint);
    \item $\bar{\lambda}_i<0$ if $i\notin\mathcal{I}$ with $\lambda_i=0$ (inactive constraint);
    \item $\bar{\lambda}_i=0$ if $i\in\mathcal{I}$ with $\lambda_i=0$ (weakly active constraint).
\end{enumerate}
1) is obvious, since if $\bar{\lambda}_i>0$ then $\lambda_i=\max\{ 0, \bar{\lambda}_i\}=\bar{\lambda}_i$. To prove 2), suppose $\mathbf{r}(\GG y,i)<\gg_i$ with $i \notin \mathcal{I}$. Then since $y=-\QQ^{-1}(\FF^\top \mu + \GG^\top \lambda + \qq)$, we have
\begin{align}\label{eq:APP_proof_2}
\mathbf{r}(-\GG \QQ^{-1}\FF^\top \mu - \GG\QQ^{-1} \GG^\top \lambda - \GG\QQ^{-1}\qq,i)<\gg_i.
\end{align}
From the definition of $\bar{\lambda}_i$ in \cref{eq:APP_lam_bar} we have
\begin{align*}
\bar{\lambda}_i&\stackrel{\hphantom{\cref{eq:APP_proof_2}}}{=} \lambda_i - \gamma \mathbf{r}(\HH_{11}\lambda+\HH_{12}\mu+\hh_1,i)\\
&\stackrel{\hphantom{\cref{eq:APP_proof_2}}}{=}\lambda_i+\gamma \mathbf{r}(-\GG \QQ^{-1}\GG^\top \lambda - \GG \QQ^{-1}\FF^\top \mu - \GG\QQ^{-1}\qq,i)-\gamma \gg_i\\
&\stackrel{\cref{eq:APP_proof_2}}{<} \lambda_i + \gamma \gg_i - \gamma \gg_i = \lambda_i = 0,
\end{align*}
since $i\notin\mathcal{I}$. This proves 2).

To prove 3), consider the case where $i\in \mathcal{I}$ and $\lambda_i=0$. Then \cref{eq:APP_proof_2} becomes
\begin{align}\label{eq:APP_proof_3}
\mathbf{r}(-\GG\QQ^{-1}\FF^\top \mu - \GG\QQ^{-1} \GG^\top \lambda - \GG\QQ^{-1}\qq,i)=\gg_i,
\end{align}
and similarly to before we have
\begin{align*}
\bar{\lambda}_i&\stackrel{\hphantom{\cref{eq:APP_proof_2}}}{=}\lambda_i+\gamma \mathbf{r}(-\GG\QQ^{-1}\GG^\top \lambda - \GG\QQ^{-1}\FF^\top \mu - \GG\QQ^{-1}\qq,i)-\gamma \gg_i\\
&\stackrel{\cref{eq:APP_proof_3}}{=} \lambda_i + \gamma \gg_i - \gamma \gg_i = \lambda_i = 0.
\end{align*}
We let $n_\text{i}$, $n_\text{wa}$, $n_\text{a}$ denote the number of inactive, weakly active, and strongly active inequality constraints, respectively. Without loss of generality, we can assume that $J_{P_C}=\operatorname{diag}(J_{\text{i}},J_\text{wa},J_\text{a},I)$, where $J_\text{i}=\operatorname{diag}(0,\dots,0)$ is associated to inactive constraints, $J_\text{wa}=\operatorname{diag}(\beta_1,\dots,\beta_{n_\text{wa}})$ is associated to weakly active constraints (with $\beta_i\in[0,1]$), and $J_\text{a}=I$ is associated to (strongly) active constraints.
Similarly, let $\GG=(\GG_\text{i},\GG_\text{wa},\GG_\text{a})$ and $v=(v_\text{i},v_\text{wa},v_\text{a},v_\text{eq})$, where $v_\text{eq}$ is associated to the equality constraints.

Assume now, for the sake of contradiction, that there exists $v\in\R^{n_\text{in}+n_\text{eq}}$, $v \neq 0$, satisfying $Uv = 0$.
This condition can be written using \cref{eq:APP_cons_jac_2} as
\begin{align}\label{eq:APP_nullspace_original}
J_{P_C}[I - \gamma \HH(\p)]v - v = 0,
\end{align}
or equivalently as in \cref{eq:APP_big_mat_1}.
Since the top left part of the $J_{P_C}$ matrix is full of zeros, the condition for $v_\text{i}$ is $v_\text{i}=0$, and we can therefore rewrite \cref{eq:APP_big_mat_1} as in \cref{eq:APP_big_mat_2}.
Next, let $\tilde{\GG}_\text{wa}:=\sqrt{\gamma}\GG_\text{wa}L$, $\tilde{\GG}_\text{a}:=\sqrt{\gamma}\GG_\text{a}L$ and $\tilde{\FF}:=\sqrt{\gamma}\FF_\text{wa}L$, where $\QQ^{-1}=LL^\top$ is the Cholesky decomposition of $\QQ^{-1}$.
With this notation \cref{eq:APP_big_mat_2} can be rewritten as in \cref{eq:APP_big_mat_3}.
We can assume that $J_\text{wa}$ is invertible, as if any $\beta_i=0$, we get the condition $\mathbf{r}(v_{\text{wa}},i)=0$ and we can treat the constraint the same way we did with the inactive constraints.
By multiplying both matrices in \cref{eq:APP_big_mat_3} by $\operatorname{diag}(J_\text{wa}^{-1},I,I)$ we obtain \cref{eq:APP_big_mat_4}.

Note that $\mathcal{A}_1$ is negative semi-definite since $J_\text{wa}^{-1}$ is a diagonal matrix with entries $1/\beta_i\in [1,+\infty)$, where $+\infty$ is not included, and $-\mathcal{A}_2$ is also negative definite since $\mathcal{A}_2=V^\top V$ where $V^\top=(\tilde{\GG}_\text{wa},\tilde{\GG}_\text{a},\tilde{\FF})$ is full row rank because of \cref{ass:DIFF_strong_convexity} and $L$ is invertible. Hence $\mathcal{A}_1-\mathcal{A}_2$ is negative definite and invertible, and the only way to verify \cref{eq:APP_nullspace_original} is to have $v=0$. We conclude that every $U$ is invertible by contradiction.
\ifBio
\else
\addtolength{\textheight}{-12cm}
\fi
\bibliographystyle{IEEEtran}
\bibliography{Sources/ref.bib}

\begin{thebibliography}{10}
\providecommand{\url}[1]{#1}
\csname url@rmstyle\endcsname
\providecommand{\newblock}{\relax}
\providecommand{\bibinfo}[2]{#2}
\providecommand\BIBentrySTDinterwordspacing{\spaceskip=0pt\relax}
\providecommand\BIBentryALTinterwordstretchfactor{4}
\providecommand\BIBentryALTinterwordspacing{\spaceskip=\fontdimen2\font plus
\BIBentryALTinterwordstretchfactor\fontdimen3\font minus \fontdimen4\font\relax}
\providecommand\BIBforeignlanguage[2]{{%
\expandafter\ifx\csname l@#1\endcsname\relax
\typeout{** WARNING: IEEEtran.bst: No hyphenation pattern has been}%
\typeout{** loaded for the language `#1'. Using the pattern for}%
\typeout{** the default language instead.}%
\else
\language=\csname l@#1\endcsname
\fi
#2}}

\bibitem{chen1998quasi}
H.~Chen and F.~Allg{\"o}wer, ``{A quasi-infinite horizon nonlinear model predictive control scheme with guaranteed stability},'' \emph{Automatica}, vol.~34, no.~10, pp. 1205--1217, 1998.

\bibitem{di2009model}
S.~Di~Cairano and A.~Bemporad, ``{Model predictive control tuning by controller matching},'' \emph{IEEE Transactions on Automatic Control}, vol.~55, no.~1, pp. 185--190, 2009.

\bibitem{soloperto2020augmenting}
R.~Soloperto, J.~K{\"o}hler, and F.~Allg{\"o}wer, ``{Augmenting MPC schemes with active learning: Intuitive tuning and guaranteed performance},'' \emph{IEEE Control Systems Letters}, vol.~4, no.~3, pp. 713--718, 2020.

\bibitem{silver2014deterministic}
D.~Silver, G.~Lever, N.~Heess, T.~Degris, D.~Wierstra, and M.~Riedmiller, ``{Deterministic policy gradient algorithms},'' in \emph{International conference on machine learning}.\hskip 1em plus 0.5em minus 0.4em\relax Pmlr, 2014, pp. 387--395.

\bibitem{amos2017optnet}
B.~Amos and J.~Z. Kolter, ``{Optnet: Differentiable optimization as a layer in neural networks},'' in \emph{International Conference on Machine Learning}.\hskip 1em plus 0.5em minus 0.4em\relax PMLR, 2017, pp. 136--145.

\bibitem{dontchev2009implicit}
A.~L. Dontchev, R.~T. Rockafellar, and R.~T. Rockafellar, \emph{{Implicit functions and solution mappings: A view from variational analysis}}.\hskip 1em plus 0.5em minus 0.4em\relax Springer, 2009, vol. 616.

\bibitem{amos2018differentiable}
B.~Amos, I.~Jimenez, J.~Sacks, B.~Boots, and J.~Z. Kolter, ``{Differentiable MPC for end-to-end planning and control},'' \emph{Advances in neural information processing systems}, vol.~31, 2018.

\bibitem{oshin2023differentiable}
A.~Oshin and E.~A. Theodorou, ``{Differentiable Robust Model Predictive Control},'' \emph{arXiv preprint arXiv:2308.08426}, 2023.

\bibitem{cortez2023robust}
W.~S. Cortez, J.~Drgona, D.~Vrabie, and M.~Halappanavar, ``{Robust Differentiable Predictive Control with Safety Guarantees: A Predictive Safety Filter Approach},'' \emph{arXiv preprint arXiv:2311.08496}, 2023.

\bibitem{gros2019data}
S.~Gros and M.~Zanon, ``{Data-driven economic NMPC using reinforcement learning},'' \emph{IEEE Transactions on Automatic Control}, vol.~65, no.~2, pp. 636--648, 2019.

\bibitem{zanon2020safe}
M.~Zanon and S.~Gros, ``{Safe reinforcement learning using robust MPC},'' \emph{IEEE Transactions on Automatic Control}, vol.~66, no.~8, pp. 3638--3652, 2020.

\bibitem{zanon2020reinforcement}
M.~Zanon, V.~Kungurtsev, and S.~Gros, ``{Reinforcement learning based on real-time iteration NMPC},'' \emph{IFAC-PapersOnLine}, vol.~53, no.~2, pp. 5213--5218, 2020.

\bibitem{bednarczuk2021lipschitz}
E.~M. Bednarczuk and K.~E. Rutkowski, ``{On Lipschitz continuity of projections onto polyhedral moving sets},'' \emph{Applied Mathematics \& Optimization}, vol.~84, no.~2, pp. 2147--2175, 2021.

\bibitem{bolte2021conservative}
J.~Bolte and E.~Pauwels, ``{Conservative set valued fields, automatic differentiation, stochastic gradient methods and deep learning},'' \emph{Mathematical Programming}, vol. 188, pp. 19--51, 2021.

\bibitem{davis2020stochastic}
D.~Davis, D.~Drusvyatskiy, S.~Kakade, and J.~D. Lee, ``{Stochastic subgradient method converges on tame functions},'' \emph{Foundations of computational mathematics}, vol.~20, no.~1, pp. 119--154, 2020.

\bibitem{bolte2021nonsmooth}
J.~Bolte, T.~Le, E.~Pauwels, and T.~Silveti-Falls, ``{Nonsmooth implicit differentiation for machine-learning and optimization},'' \emph{Advances in neural information processing systems}, vol.~34, pp. 13\,537--13\,549, 2021.

\bibitem{agrawal2020learning}
A.~Agrawal, S.~Barratt, S.~Boyd, and B.~Stellato, ``{Learning convex optimization control policies},'' in \emph{Learning for Dynamics and Control}.\hskip 1em plus 0.5em minus 0.4em\relax PMLR, 2020, pp. 361--373.

\bibitem{agrawal2019differentiable}
A.~Agrawal, B.~Amos, S.~Barratt, S.~Boyd, S.~Diamond, and J.~Z. Kolter, ``{Differentiable convex optimization layers},'' \emph{Advances in neural information processing systems}, vol.~32, 2019.

\bibitem{diehl2005real}
M.~Diehl, H.~G. Bock, and J.~P. Schl{\"o}der, ``A real-time iteration scheme for nonlinear optimization in optimal feedback control,'' \emph{SIAM Journal on control and optimization}, vol.~43, no.~5, pp. 1714--1736, 2005.

\bibitem{bolte2022differentiating}
J.~Bolte, E.~Pauwels, and A.~J. Silveti-Falls, ``{Differentiating nonsmooth solutions to parametric monotone inclusion problems},'' \emph{arXiv preprint arXiv:2212.07844}, 2022.

\bibitem{zuliani2024closed}
R.~Zuliani, E.~C. Balta, and J.~Lygeros, ``{Closed-loop Performance Optimization of Model Predictive Control with Robustness Guarantees},'' \emph{arXiv preprint arXiv:2403.04655}, 2024.

\bibitem{coste1999introduction}
M.~Coste, \emph{Introduction to o-minimal geometry}.\hskip 1em plus 0.5em minus 0.4em\relax Rennes, France: Institut de recherche math{\'e}matique de Rennes (IRMAR), 1999.

\bibitem{bolte2021nonsmooth_extended}
\BIBentryALTinterwordspacing
J.~Bolte, T.~Le, E.~Pauwels, and A.~Silveti{-}Falls, ``Nonsmooth implicit differentiation for machine learning and optimization,'' \emph{CoRR}, vol. abs/2106.04350, 2021. [Online]. Available: \url{https://arxiv.org/abs/2106.04350}
\BIBentrySTDinterwordspacing

\bibitem{wang2009fast}
Y.~Wang and S.~Boyd, ``Fast model predictive control using online optimization,'' \emph{IEEE Transactions on control systems technology}, vol.~18, no.~2, pp. 267--278, 2009.

\bibitem{beck2017first}
A.~Beck, \emph{{First-order methods in optimization}}.\hskip 1em plus 0.5em minus 0.4em\relax SIAM, 2017.

\bibitem{bauschke2017convex}
H.~H. Bauschke, P.~L. Combettes, H.~H. Bauschke, and P.~L. Combettes, \emph{{Convex Analysis and Monotone Operator Theory in Hilbert Spaces}}.\hskip 1em plus 0.5em minus 0.4em\relax Springer, 2017.

\bibitem{coste2000introduction}
M.~Coste, ``{An introduction to semialgebraic geometry},'' 2000.

\bibitem{aboudonia2020distributed}
A.~Aboudonia, A.~Eichler, and J.~Lygeros, ``{Distributed model predictive control with asymmetric adaptive terminal sets for the regulation of large-scale systems},'' \emph{IFAC-PapersOnLine}, vol.~53, no.~2, pp. 6899--6904, 2020.

\bibitem{kerrigan2000soft}
E.~C. Kerrigan and J.~M. Maciejowski, ``{Soft constraints and exact penalty functions in model predictive control},'' in \emph{UKACC International Conference (Control 2000), Cambridge}, 2000.

\bibitem{burke1991exact}
J.~V. Burke, ``{An exact penalization viewpoint of constrained optimization},'' \emph{SIAM Journal on control and optimization}, vol.~29, no.~4, pp. 968--998, 1991.

\bibitem{andersson2019casadi}
J.~A. Andersson, J.~Gillis, G.~Horn, J.~B. Rawlings, and M.~Diehl, ``{CasADi: a software framework for nonlinear optimization and optimal control},'' \emph{Mathematical Programming Computation}, vol.~11, pp. 1--36, 2019.

\bibitem{arnstrom2022dual}
D.~Arnstr{\"o}m, A.~Bemporad, and D.~Axehill, ``A dual active-set solver for embedded quadratic programming using recursive ldl updates,'' \emph{IEEE Transactions on Automatic Control}, vol.~67, no.~8, pp. 4362--4369, 2022.

\bibitem{guemghar2002predictive}
K.~Guemghar, B.~Srinivasan, P.~Mullhaupt, and D.~Bonvin, ``{Predictive control of fast unstable and nonminimum-phase nonlinear systems},'' in \emph{Proceedings of the 2002 American Control Conference (IEEE Cat. No. CH37301)}, vol.~6.\hskip 1em plus 0.5em minus 0.4em\relax IEEE, 2002, pp. 4764--4769.

\bibitem{biegler2009large}
L.~T. Biegler and V.~M. Zavala, ``{Large-scale nonlinear programming using IPOPT: An integrating framework for enterprise-wide dynamic optimization},'' \emph{Computers \& Chemical Engineering}, vol.~33, no.~3, pp. 575--582, 2009.

\bibitem{verschueren2022acados}
R.~Verschueren, G.~Frison, D.~Kouzoupis, J.~Frey, N.~v. Duijkeren, A.~Zanelli, B.~Novoselnik, T.~Albin, R.~Quirynen, and M.~Diehl, ``acados—a modular open-source framework for fast embedded optimal control,'' \emph{Mathematical Programming Computation}, vol.~14, no.~1, pp. 147--183, 2022.

\bibitem{wachsmuth2013licq}
G.~Wachsmuth, ``{On LICQ and the uniqueness of Lagrange multipliers},'' \emph{Operations Research Letters}, vol.~41, no.~1, pp. 78--80, 2013.

\end{thebibliography}
\end{document}